\documentclass[oneside]{amsart}

\usepackage{amsmath,amssymb,amsthm}
\usepackage[width=0.75\paperwidth]{geometry}
\usepackage[
    colorlinks=true,
    linkcolor=red,
    citecolor=green,
    urlcolor=blue
]{hyperref}

\newcommand{\E}{\mathbb{E}}
\newcommand{\R}{\mathbb{R}}
\newcommand{\bV}{\mathbf{V}}
\newcommand{\cW}{\mathcal{W}}

\newcommand{\Id}{\operatorname{Id}}
\newcommand{\p}{\partial}
\newcommand{\ud}{\,\mathrm{d}}

\numberwithin{equation}{section}

\newtheorem{theo}{Theorem}[section]
\newtheorem{prop}[theo]{Proposition}
\newtheorem{lem}[theo]{Lemma}
\newtheorem{cor}[theo]{Corollary}
\newtheorem{exa}[theo]{Example}
\newtheorem{rem}[theo]{Remark}
\begin{document}

\author[X Feng]{Xuanrui Feng}
\address{School of Mathematical Sciences, Peking University, Beijing 100871, China.}
\email{pkufengxuanrui@stu.pku.edu.cn}

\author[C Liu]{Chenguang Liu}
\address{Technische Universiteit Delft, Mekelweg 5, 2628 CD Delft, Netherlands}
\email{liucg92@gmail.com}

\author[Z Wang]{Zhenfu Wang}
\address{Beijing International Center for Mathematical Research, Peking University, Beijing 100871, China}
\email{zwang@bicmr.pku.edu.cn}

\title[Uniform-in-time propagation of chaos]
{Uniform-in-time relative entropy estimates for Kac's approximation of the Landau equation}
\date{\today}

\begin{abstract}
    We prove uniform-in-time quantitative relative entropy estimates for Kac's particle approximation of the three-dimensional spatially homogeneous Landau equation for Maxwellian molecules. Using the relative entropy method and logarithmic derivative estimates, we obtain a finite-time normalized entropy bound of order $N^{-1/2}$. The proof relies on explicit covariance identities and a new law-of-large-numbers estimate for the particle system using a new duality method. The uniform-in-time propagation of chaos is derived from an interpolation with an algebraic relaxation estimate toward equilibrium. 
\end{abstract}

\maketitle
\enlargethispage{6pt}
\tableofcontents

\section{Introduction}

\subsection{Landau equation for Maxwellian molecules}

Consider the spatially homogeneous Landau equation for Maxwellian molecules in dimension three:
\begin{equation}\label{eq:landau-equation}
    \p_t f=\int_{\R^3} (\nabla_v-\nabla_w) \cdot \Big( a(v-w) \cdot (\nabla_v-\nabla_w) (f(v)f(w) ) \Big) \ud w,
\end{equation}
where the coefficient matrix $a(z)=|z|^2 \Id-z \otimes z$ is non-negative definite. 
Using convolution expressions and integration by parts, we rewrite it into the more commonly used divergence form:
\begin{equation}\label{eq: div-landau-equation}
    \p_t f=\nabla \cdot \Big( (a \ast f) \cdot \nabla f-(b \ast f)f \Big),
\end{equation}
or non-divergence form:
\begin{equation*}
    \p_t f=(a \ast f) : \nabla^2 f-(c \ast f)f,
\end{equation*}
where a direct computation gives the vector $b(z)=\nabla \cdot a(z)=-2z$ and the scalar $c(z)=\nabla \cdot b(z)=-6$. Denote the initial density by $f_0$. We normalize the initial data by assuming
\begin{equation}\label{eq:normalization-condition}
    \int_{\R^3} f_0 (v) \ud v=1, \quad \int_{\R^3} vf_0 (v) \ud v=0, \quad  \int_{\R^3} |v|^2 f_0 (v) \ud v =3,
\end{equation}
and by the conservation of mass, momentum and kinetic energy, these conditions also hold for any $t \geq 0$. For convenience, we define the $k$-th order moment of $f_0$:
\begin{align*}
    M_k= \int_{\R^3} |v|^kf_0(v)\ud v.
\end{align*}
 We also define the covariance matrix of $f$ at any time $t$:
\begin{equation*}
    E(t)=\int_{\R^3} v \otimes v f_t(v) \ud v.
\end{equation*}
After choosing an orthogonal basis that diagonalizes the symmetric matrix $E(0)$, we write
\begin{equation*}
    E(0)=\Id+\operatorname{diag}(D_1,D_2,D_3).
\end{equation*}
The energy normalization implies that $\displaystyle \sum_{\alpha=1}^3 D_\alpha=0$ and $D_\alpha \in [-1,2]$. A direct computation shows that the covariance matrix has the explicit form
\begin{equation*}
    E(t)=\Id+e^{-12t}\operatorname{diag}(D_1,D_2,D_3),
\end{equation*}
and we refer to Villani \cite{villani1998spatially} for this classical result. With these assumptions, we compute the explicit coefficients in the equation:
\begin{equation*}
    a \ast f(v)=(|v|^2\Id-v \otimes v)+(3\Id-E(t)), \quad b \ast f(v)=-2v, \quad c \ast f(v)=-6.
\end{equation*}
The well-posedness of the Landau equation for Maxwellian molecules was established in Villani \cite{villani1998spatially}. More generally, one considers a wider range of potentials
\begin{equation*}
    a_\gamma(z)=|z|^\gamma (|z|^2 \Id-z \otimes z), \quad \gamma\in[-3,1].
\end{equation*}
The critical exponent $\gamma=-3$ corresponds to Coulomb interactions in Landau's plasma model \cite{landau1936kinetische}. The ranges $\gamma\in(0,1]$, $\gamma\in[-2,0)$, and $\gamma\in[-3,-2)$ are usually called hard, moderately soft, and very soft potentials, respectively. Relevant regularity and well-posedness results include \cite{desvillettes2000spatially1,desvillettes2000spatially2,silvestre2017upper,guillen2025landau2}. Here we restrict attention to Maxwellian molecules, $\gamma=0$, whose explicit covariance dynamics is essential to our argument.

\subsection{Kac's system for the Landau equation}

We are interested in the derivation of the Landau equation from a many-particle system. In this article, our starting point is Kac's system for the Landau equation, described by the following Landau master equation
\begin{equation}\label{eq:master-equation}
    \p_t F_N=\frac{1}{N} \sum_{1 \leq i<j \leq N} (\nabla_{v^i}-\nabla_{v^j}) \cdot \Big( a(v^i-v^j) \cdot (\nabla_{v^i}-\nabla_{v^j}) F_N \Big).
\end{equation}
The initial value is chosen to be fully factorized: $F_N(0)=f_0^{\otimes N}$. The weak formulation of the Landau master equation is given by
\begin{equation*}
    \begin{aligned}
        \p_t \langle F_N, \varphi \rangle&=\frac{1}{N}\sum_{i,j=1}^N \int_{\R^{3N}} b(v^i-v^j) \cdot (\nabla_{v^i} \varphi -\nabla_{v^j} \varphi) F_N \ud V\\
        &+\frac{1}{2N} \sum_{i,j=1}^N \int_{\R^{3N}} a(v^i-v^j):
        \Big(\nabla_{v^iv^i}^2 \varphi-\nabla_{v^iv^j}^2 \varphi-\nabla_{v^jv^i}^2 \varphi+\nabla_{v^jv^j}^2 \varphi\Big)F_N \ud V,
    \end{aligned}
\end{equation*}
where $\ud V=\ud v^1 \cdots \ud v^N$ and $\varphi \in C_b^2(\R^{3N})$ is any test function whose derivatives up to second order are continuous and bounded. 

The master equation has the following conservative stochastic representation. Let $V^i(t)$ be the velocity of particle $i$. Consider the particle system whose evolution is governed by the following SDE:
\begin{equation}\label{eq:particle-sde}
    \ud V^i(t)=\frac{2}{N}\sum_{j=1}^N b \left( V^i(t)-V^j(t) \right) \ud t +\frac{\sqrt{2}}{\sqrt{N}} \sum_{j=1}^N \sigma \left( V^i(t)-V^j(t) \right) \ud B_t^{i,j}.
\end{equation}
The particle system was first introduced by Carrapatoso \cite{carrapatoso2016propagation}. Here the diffusion coefficient is given by a square root matrix of $a(z)$:
\begin{equation*}
    \sigma(z)=|z|\Pi(z), \quad \Pi(z)=\Id-\frac{z\otimes z}{|z|^2} \quad (z\ne0), \quad \sigma(0)=0.
\end{equation*}
Thus $\sigma(z)\sigma(z)^\top=a(z)$ and $\sigma(-z)=\sigma(z)$. For $i<j$, the processes $B^{i,j}$ are independent three-dimensional standard Brownian motions, and let $B^{j,i}=-B^{i,j}$ and $B^{i,i}=0$ in the sense of antisymmetry. By It\^o's formula, the joint law of $\mathbf{V}(t)= (V^1(t),\cdots,V^N(t))$ satisfies the Landau master equation \eqref{eq:master-equation}. Moreover, the system \eqref{eq:particle-sde} conserves total momentum and kinetic energy pointwise. The corresponding normalized quantities are defined by
\begin{equation}\label{eq:particle-conservation}
    m_N=\frac{1}{N} \sum_{i=1}^N V^i(t)=\frac{1}{N} \sum_{i=1}^N V^i(0), \quad e_N=\frac{1}{N} \sum_{i=1}^N |V^i(t)|^2=\frac{1}{N} \sum_{i=1}^N |V^i(0)|^2.
\end{equation}
These identities are used repeatedly in the law-of-large-numbers estimates below and hold almost surely. 

\subsection{Main results}

Denote by $F_N(t)$ the joint law of Kac's system \eqref{eq:master-equation} at time $t$, and  by $F_{N,k}(t)$ its $k$-th order marginal. We write $f_t^{\otimes N}$ for the tensorization of the solution $f_t$ of the Landau equation \eqref{eq:landau-equation}. The normalized relative entropy is defined by
\begin{equation*}
    H_N \left( F_N(t)|f_t^{\otimes N} \right)=\frac{1}{N} \int_{\R^{3N}} F_N(t) \log \frac{F_N(t)}{f_t^{\otimes N}} \ud V.
\end{equation*}

Our first main result is a quantitative finite time propagation of chaos estimate in the sense of relative entropy. Unless otherwise stated, constants denoted by $C$ may change from line to line and depend on the initial constants and norms but are independent of $N$ and $t$. The explicit dependence is given in the theorem.

\begin{theo}[Finite time relative entropy estimate]\label{the:finite-time-entropy}
    Let $F_N$ be an entropy solution to \eqref{eq:master-equation}, and let $f$ be the smooth classical solution to \eqref{eq:landau-equation}, with initial data $F_N(0)=f_0^{\otimes N}$ for some probability density $f_0 \in L^1 \cap L^\infty(\R^3)$.
    Assume that $f_0$ has a Gaussian upper bound
    \begin{equation}\label{con:gaussian-upper-bound}
        f_0(v) \leq L_0 e^{-L_0^{-1}|v|^2}
    \end{equation}
    for some constant $L_0$, and satisfies the logarithmic derivative estimates
    \begin{equation}\label{con:logarithmic-derivative}
        |\nabla^j \log f_0(v)| \leq L_j(1+|v|)^j
    \end{equation}
    for $j=1,2,3$ for some constants $L_1,L_2,L_3$. Then for any $T>0$, there exists a constant $M>0$ which depends on the constants $L_0, L_1, L_2, L_3$ but is independent of $N$ and $T$, such that
    \begin{equation*}
        \sup_{0 \leq t \leq T} H_N \left( F_N(t)| f_t^{\otimes N} \right) \leq \frac{M(1+T)^5}{\sqrt N}.
    \end{equation*}
\end{theo}

The proof of Theorem \ref{the:finite-time-entropy} starts with the standard time evolution of relative entropy computation, which is similar to that in \cite{carrillo2025relative}. However, because the master equation for Kac's system differs from that of the Nanbu-type particle system, we need to establish a new law-of-large-numbers estimate for test functions involving the score function $\nabla \log f$ and the Landau coefficients. Since the score function has no explicit polynomial structure, we cannot explicitly expand the relevant terms or compute their exact time evolution as in \cite{carrillo2025relative}.

The main technical contributions in this article are summarized in Lemma \ref{lem:polynomial_backward_estimates} and Lemma \ref{lem:score_weighted_moment_lln}. We develop a duality method at the level of a dual semigroup of the limiting equation, which has a nice structure due to the regularity of $f$. Therefore, we transfer the law-of-large-numbers estimate at time $t$ with a given test function $\psi$ to that estimate at initial time $t=0$ with a time-dependent backward solution $u_s$ as the test function. The result then follows from the initial chaotic assumption and the nice properties of $u_s$, namely the polynomial growth estimate. This also requires a proof of logarithmic derivative estimates up to third order, and we  prove it in Theorem \ref{the:logarithmic_derivative} up to any finite order as a key technical novelty. We expect that our new duality method would be of wide use in quantitative propagation of chaos for singular systems. 

\begin{rem}
    The first-order derivative bound in \eqref{con:logarithmic-derivative} implies a Gaussian lower bound for $f_0$, see also \cite[Proposition 3.2]{carrillo2025relative} for instance. More precisely, there exists a constant $c_0>0$ depending on $L_1$ such that 
    \begin{equation}\label{eq:initial-gaussian-bounds}
        f_0(v) \geq c_0 e^{-c_0^{-1} |v|^2}.
    \end{equation}
    By \cite[Section 7]{villani1998spatially}, such a Gaussian lower bound holds for any time $t \geq 0$, under possible modification of the positive constant $c_0>0$. For simplicity, we use \eqref{eq:initial-gaussian-bounds} for any time $t \geq 0$.
    Moreover, \eqref{con:logarithmic-derivative} also implies a finite Fisher information for $f_0$: 
    \begin{align*}
        \int_{\R^3} f_0 |\nabla \log f_0(v)|^2 \ud v\le L_1 \int_{\R^3} (1+|v|^2) f_0(v)  \ud v <\infty.
    \end{align*}
\end{rem}

\begin{rem}
    The time dependence in Theorem \ref{the:finite-time-entropy} has polynomial growth of order $T^5$. This growth rate comes from the logarithmic estimates in Theorem \ref{the:logarithmic_derivative} below and repeated use of the Cauchy--Schwarz inequality in time integral. A more careful computation or a preservation of the exponential term in the estimate of $|D_v X_r^{s,v}|$ in Step 2 of Lemma \ref{lem:polynomial_backward_estimates} may provide a better rate. We do not pursue an optimization of this rate here, since nevertheless we would obtain a uniform-in-time convergence easily in the next theorems.
\end{rem}

Our second main result is to upgrade the finite-time convergence result to a uniform-in-time version. This is formulated by the following uniform-in-time relative entropy estimate.

\begin{theo}[Uniform-in-time relative entropy estimate]\label{the:time_uniform}
    Assume the hypotheses of Theorem \ref{the:finite-time-entropy}. Assume in addition that $f_0$ is $L^2$ relative to the standard Gaussian distribution:
    \begin{equation}\label{eq:relative_L2_initial}
        A_0=\int_{\R^3} \left( \frac{f_0}{\gamma}-1 \right)^2 \gamma \ud v<\infty, \quad \gamma(v)=(2\pi)^{-3/2} e^{-|v|^2/2}.
    \end{equation}
    Then for any $\delta>0$, there exists a constant $C_\delta>0$, depending on $\delta$ and the initial bounds $A_0, L_0, L_1, L_2, L_3$, such that
    \begin{equation*}
        H_N \left( F_N(t)|f_t^{\otimes N} \right) \leq C_\delta \left( \frac{1}{(1+t)^\delta}+\frac{1}{\sqrt N} \right),
    \end{equation*}
    for any $t \geq 0$.
\end{theo}

Interpolating between Theorem \ref{the:finite-time-entropy} and Theorem \ref{the:time_uniform}, we deduce uniform-in-time propagation of chaos in the sense of relative entropy.

\begin{cor}[Uniform-in-time propagation of chaos]\label{cor:uniform-in-time-poc}
    Under the assumptions of Theorem \ref{the:time_uniform}, for any $0<\varepsilon<1/2$, there exists a constant $C_{\varepsilon}>0$, depending on $\varepsilon$ and the initial bounds $A_0, L_0, L_1, L_2, L_3$, such that
    \begin{equation}\label{ineq: entropy uni}
        \sup_{t \geq 0} H_N \left( F_N(t)|f_t^{\otimes N} \right) \leq  C_\varepsilon N^{-1/2+\varepsilon}.
    \end{equation}
     Moreover, for the particle system SDE \eqref{eq:particle-sde}, for any $0<\varepsilon<1/4$, we have
    \begin{equation}\label{ineq: was2}
        \sup_{t \geq 0} \E \bigg[ \mathcal{W}_1 \bigg( \frac 1 N \sum_{i=1}^N \delta_{V^i(t)},f_t \bigg) \bigg] \leq C_\varepsilon N^{-\frac{1}{4}+\varepsilon},
    \end{equation}
    where $\mathcal{W}_1$ is the Wasserstein-1 distance on $\mathcal{P}_2(\R^3)$.
\end{cor}

\begin{rem}
    The uniform-in-time quantitative convergence rate in Corollary \ref{cor:uniform-in-time-poc} could be compared with the estimate of Fournier--Guillin \cite[Theorem 4]{fournier2017kac}, who proved uniform-in-time propagation of chaos in Wasserstein-2 distance: for any $0<\varepsilon<1/3$, there exists a constant $C_\varepsilon>0$ depending on $\varepsilon$ and the initial entropy and high-order moments, such that
    \begin{equation*}
        \sup_{t \geq 0} \E \bigg[ \cW_2^2\left( \frac 1 N  \sum_{i=1}^N \delta_{V^i(t)}, f_t \right) \bigg] \leq C_\varepsilon N^{-1/3+\varepsilon}.
    \end{equation*}
    If $f_t$ satisfies a Talagrand $T_2$ inequality with a constant uniform in time (see for instance Villani \cite{villani2021topics}), then our relative entropy estimate would imply such a convergence result in Wasserstein-2 distance with a better rate $N^{-1/2+ \varepsilon}$. Nevertheless, we need to impose stronger initial assumptions \eqref{con:gaussian-upper-bound} and \eqref{con:logarithmic-derivative}, which are used primarily in the proof of our new law-of-large-numbers estimate Proposition \ref{prop:weighted_lln_ap} via our duality argument. On the other hand, if we use the estimate in \cite{fournier2017kac} in place of Proposition \ref{prop:weighted_lln_ap}, the initial assumptions could be relaxed, but the convergence rate will be  worse than our current result.
\end{rem}

\subsection{Discussion on the Nanbu-type system and the duality method}

The program of particle approximation for the Landau equation has been well studied in the past decade. Apart from Kac's system \eqref{eq:particle-sde} studied in this article, we also mention the Nanbu-type particle system introduced by Fournier \cite{fournier2009particle} in a probabilistic sense. The SDE reads
\begin{equation*}
    \ud V^i(t)=\frac{2}{N} \sum_{j=1}^N b \left( V^i(t)-V^j(t) \right) \ud t+\sqrt{2} \bigg( \frac{1}{N} \sum_{j=1}^N a \left( V^i(t)-V^j(t) \right) \bigg)^{1/2} \ud B_t^i, \quad 1 \leq i \leq N,
\end{equation*}
where $B^i$ are independent standard Brownian motions. The corresponding master equation reads
\begin{equation*}
    \partial_t F_N=\sum_{i=1}^N \nabla_{v^i} \cdot \bigg[ \bigg( \frac{1}{N} \sum_{j=1}^N a(v^i-v^j) \bigg) \cdot \nabla_{v^i} F_N-\bigg( \frac{1}{N} \sum_{j=1}^N b(v^i-v^j) \bigg) F_N \bigg],
\end{equation*}
which is different from \eqref{eq:master-equation}. In particular, the Nanbu-type system only preserves momentum and energy in expectation, while Kac's system preserves them pathwise.

The mean-field convergence from the Nanbu-type system to the Landau equation with Maxwellian molecules was studied in \cite{carrillo2025relative} using also the relative entropy method, which in general requires a law-of-large-numbers estimate in the form of
\begin{equation*}
    \E_{F_N(t)} \left| \langle\psi , \mu_N(t)-f_t \rangle \right|^2.
\end{equation*}
In \cite{carrillo2025relative} the test functions are of polynomial forms  with degrees less than or equal to $2$, while in this article we need to deal with various test functions containing the coefficients $a,b$ and the score function $\nabla \log f$, see Proposition \ref{prop:weighted_lln_ap} for details. When $\psi$ is purely polynomial function, the convolution terms can be expanded into explicit polynomials, which allows an explicit computation of time evolution of each term in \cite{carrillo2025relative}. In this article, however, since the score function $\nabla \log f$ has no explicit structure,  a direct computation fails, which inspires us to introduce the duality method in Section \ref{sec:lln} to push everything back to the initial time. This is the main technical contribution of this article.

For the Nanbu-type system, since there is no pathwise conservation of momentum and energy, the computation of time evolution of moments of every order in \cite{carrillo2025relative} is quite complicated. In this subsection, we briefly sketch an alternative duality method to simplify the explicit computation and provide a new point of view for such estimates.

For any test function $\psi$, we define the bilinear form operator
\begin{equation*}
    \mathcal{B}_\psi(\mu, \nu)=\iint K_\psi (v,w) \mu(\ud v) \nu( \ud w)
\end{equation*}
for any signed measures $\mu$ and $\nu$, where the kernel function
\begin{equation*}
    K_\psi(v,w)=a(v-w):\nabla^2 \psi(v)+2b(v-w) \cdot \nabla \psi(v).
\end{equation*}
With these notations, we have
\begin{equation*}
    \frac{\ud}{\ud t} \E_{F_N(t)} \langle \mu_N(t), \psi \rangle=\E_{F_N(t)} \mathcal{B}_\psi (\mu_N(t), \mu_N(t)), \quad \frac{\ud}{\ud t} \langle f_t, \psi \rangle=\mathcal{B}_\psi (f_t,f_t).
\end{equation*}
Now we define the dual backward equation by
\begin{equation*}
    \partial_s u_s+\mathcal{L}_f^\ast u_s=0, \quad u_t=\psi,
\end{equation*}
where the linear operator $\mathcal{L}_f^\ast$ is given by
\begin{equation*}
    (\mathcal{L}_f^\ast \phi)(v)=\int K_\phi(v,w) f(w) \ud w+\int K_\phi(w,v) f(w) \ud w.
\end{equation*}
This well-constructed dual backward solution $u_s$ preserves the quadratic form of the test function $\psi$ in the following sense: if $\psi(v)=v^{\mathsf{T}} Hv$ for some symmetric constant matrix $H$, then there exists a decomposition
\begin{equation*}
    u_s(v)=v^{\mathsf{T}} H_s v+c_s,
\end{equation*}
where $H_s$ is a time-dependent symmetric constant matrix with $|H_s| \leq C|H|$ and $c_s$ is a constant. The explicit computation follows from a straightforward computation and is omitted here since it is not the main goal of this article.

With this explicit formula, we can proceed as in Lemma \ref{lem:score_weighted_moment_lln} of this article by rewriting the target term by duality:
\begin{equation*}
    R_\psi^N(t)=\langle \psi, \mu_N(t)-f_t \rangle=\langle u_0, \mu_N(0)-f_0 \rangle+\int_0^t \ud \langle u_s, \mu_N(s)-f_s \rangle.
\end{equation*}
The first term is estimated at the initial time $t=0$ and the law-of-large-numbers result comes from the initial chaotic assumption, and the second term is estimated by It\^o's formula and explicit decomposition of $u_s$. This routine therefore somehow unifies the complicated computation for various moment terms in \cite{carrillo2025relative} into one single duality framework.

\subsection{Related works}

The program of deriving kinetic equations from stochastic particle systems via mean-field limits was initiated by Kac in his seminal article \cite{kac1956foundations}, where he studied his caricature of the one-dimensional Boltzmann equation. The classical probabilistic formulation of propagation of chaos can be found for instance in McKean \cite{mckean1967propagation} and Sznitmann \cite{sznitman1991topics}, and more recent progress on Kac's program was shown  for instance in Mischler--Mouhot \cite{mischler2013kac} and Hauray--Mischler \cite{hauray2014kac}. 

Entropy, Fisher information, and relative entropy serve as central tools  for proving singular mean-field limits and propagation of chaos recently. At the qualitative level, Hauray and Mischler \cite{hauray2014kac} developed the theory of entropic and Fisher-information chaos, while Fournier, Hauray, and Mischler \cite{fournier2014propagation} used entropy and Fisher-information compactness to establish propagation of chaos for singular stochastic vortex systems. At the quantitative level, Jabin and Wang \cite{jabin2016mean,jabin2018quantitative} developed a relative-entropy method that yields explicit propagation-of-chaos estimates for stochastic systems with bounded or $W^{-1,\infty}$ interaction kernels. Carrillo, Feng, Guo, Jabin, and Wang \cite{carrillo2025relative} subsequently adapted this method to the Landau particle system for Maxwellian molecules, directly comparing its $N$-particle law with the tensorized solution of the Landau equation. Uniform-in-time propagation-of-chaos estimates for singular stochastic systems have also been established; see, for example, Guillin, Le Bris, and Monmarch\'e \cite{guillin2024uniform}.

For the spatially homogeneous Landau equation, particle approximations and propagation of chaos have been obtained by several complementary methods. Early probabilistic and numerical particle approximations include Fontbona--Gu\'erin--M\'el\'eard \cite{fontbona2009measurability} and Fournier \cite{fournier2009particle}; compactness and hierarchical arguments for Kac's model for the Landau equation go back in particular to Miot--Pulvirenti--Saffirio \cite{miot2011kac}. See also the work by Carrillo--Guo \cite{carrillo2025fisher}. For Maxwellian molecules, Carrapatoso \cite{carrapatoso2016propagation} proved quantitative uniform-in-time propagation of chaos in the Wasserstein-$1$ distance, as well as entropic chaos for the dynamics on the centered fixed-energy Boltzmann sphere. For hard potentials and Maxwellian molecules, Fournier--Guillin \cite{fournier2017kac} obtained quantitative convergence from a Kac-like particle system to the Landau equation, while for moderately soft potentials the qualitative and partially quantitative theory was developed in Fournier--Hauray \cite{fournier2016propagation}.

The regimes of very soft potentials and Coulomb interactions require substantially different ideas because of the strong singularity and degeneracy of the Landau diffusion matrix. Recent progress was made by Feng--Wang \cite{feng2026kac}, who treat the full range of interaction potentials, including the Coulomb interactions, by a duality and cluster-expansion approach together with new functional estimates. In parallel, Tabary \cite{tabary2025propagation} treated very soft potentials and Coulomb interactions by a compactness-uniqueness method based on entropy production and higher-order Fisher-information estimates.

\subsection{Outline of the article}

The rest of this article is organized as follows. In Section \ref{sec:logarithmic_derivatives} we prove a new logarithmic derivative estimate of any finite order, which is one of the key technical novelties of this article. Section \ref{sec:entropy_evolution} derives the evolution of the relative entropy, which reduces the proof of our main theorems into the proof of several law of large numbers estimates. We also present some preliminary estimates. Section \ref{sec:lln} proves a new law-of-large-numbers estimate for a large class of unbounded test functions by introducing a new duality method, which is another technical contribution of this article and potentially of wider use. In Section \ref{sec:finite_time} we prove the finite-time relative entropy estimate in Theorem \ref{the:finite-time-entropy}. The proof of the uniform-in-time estimate Theorem \ref{the:time_uniform} is completed in Section \ref{sec:uniform_time}, by a combination of an HWI inequality, a relaxation to equilibrium estimate on the Boltzmann sphere, and propagation of exponential moment estimates.

\section{Logarithmic derivatives of the Landau equation}\label{sec:logarithmic_derivatives}

In this section, we derive an improved logarithmic derivative estimate for the Landau equation \eqref{eq:landau-equation} up to any finite order, which is an extension of the gradient and Hessian estimates in \cite{carrillo2025relative}. In particular, we adapt a decomposition of the Landau operator so that we treat the derivative formulas in a unified way, which simplifies the computation in \cite{carrillo2025relative}.

We introduce the following decomposition of the Landau operator, which is similar to \cite{villani1998spatially}. Define
\begin{equation*}
    B(t)=\operatorname{diag} \left( 2-e^{-12t} D_1, 2-e^{-12t} D_2, 2-e^{-12t} D_3 \right).
\end{equation*}
The non-degeneracy assumption of the initial data gives $B(t) \geq \eta \Id$ for some $\eta>0$.  For $1 \leq p<q \leq 3$, we define the differential operators
\begin{equation*}
    \Omega_{pq}=v_p \partial_{v_q}-v_q \partial_{v_p}.
\end{equation*}
Let $A_{pq}=e_q \otimes e_p-e_p \otimes e_q$ be the corresponding antisymmetric constant matrices.  Then we have $\Omega_{pq}=(A_{pq}v) \cdot \nabla$. Using $\displaystyle \sum_{p<q}\Omega_{pq}^2=(|v|^2 \Id-v \otimes v):\nabla^2-2v \cdot \nabla$, the Landau equation \eqref{eq:landau-equation} takes the form
\begin{equation}\label{eq:landau_rotation_decomposition}
    \partial_t f=Lf=\bigg( B(t):\nabla^2+\sum_{p<q} \Omega_{pq}^2+2v \cdot \nabla+6 \bigg) f.
\end{equation}

\begin{theo}[Logarithmic derivative estimate]
\label{the:logarithmic_derivative}
    Let $f$ be the smooth classical solution of \eqref{eq:landau-equation} with the normalizations above. Assume that $f_0$ has a Gaussian upper bound
    \begin{equation*}
        f_0(v) \leq L_0 e^{-L_0^{-1}|v|^2}
    \end{equation*}
    and logarithmic derivative bounds up to order $m$:
    \begin{equation*}
        |\nabla^k \log f_0(v)| \leq L_k (1+|v|^k), \quad 1 \leq k \leq m.
    \end{equation*}
    Then we have
    \begin{equation*}
        |\nabla^k\log f(t,v)| \leq C_k(1+t+|v|^2)^{k/2}, \quad 1 \leq k \leq m,
    \end{equation*}
    where $C_k$'s depend only on $L_k$'s and $\eta$.
\end{theo}

Here we state this estimate in a general setting. In particular, for the proof of the main result of this article, it suffices to take $m=3$ here.

The brief idea follows from Bernstein's method. We construct some auxiliary functions which typically involve $\displaystyle \frac{|\nabla^k f|^2}{f}$ and $f (\log f)^k$. In the spirit of $\Gamma$-calculus, we show that these auxiliary functions remain non-positive by the maximum principle in \cite{carrillo2025relative} and the desired estimates follow. We first prove two technical lemmas.

\begin{lem}[$\Gamma$-calculus of high-order derivatives]\label{lem:higher_order_quotient}
    Denote by $\displaystyle E_k=\frac{|\nabla^k f|^2}{f}$. For any $1 \leq k \leq m$, we have
    \begin{equation*}
        (\partial_t-L) E_k \leq 4k E_k,
    \end{equation*}
    and a refined version
    \begin{equation*}
        (\partial_t-L) E_k \leq 4k E_k-\eta E_{k+1}+C|\nabla \log f|^2 E_k.
    \end{equation*}
\end{lem}

\begin{proof}
    The starting point is to derive the equation solved by $\nabla^k f$. This requires the computation of commutators between the operators in the bracket in \eqref{eq:landau_rotation_decomposition} and $\nabla^k$. Obviously $[\nabla^k,v\cdot\nabla]=k\nabla^k$ and the constant multiplication commutes with $\nabla^k$.

    Recalling $\Omega_{pq}=(A_{pq}v) \cdot \nabla$ with $A_{pq}=e_q \otimes e_p-e_p \otimes e_q$, we differentiate one component to get
    \begin{equation*}
        \partial_i (\Omega_{pq}f) =\Omega_{pq} (\partial_i f)+(A_{pq})_{\ell i} \partial_\ell f.
    \end{equation*}
    We iterate the differentiation and obtain
    \begin{equation}\label{eq:rotation_tensor_commutator}
        \nabla^k(\Omega_{pq}f)=\Omega_{pq} \nabla^k f+R_{pq}^{(k)} \nabla^k f,
    \end{equation}
    where $R_{pq}^{(k)}$ applies $A_{pq}^{\mathsf T}=-A_{pq}$ in turn to each of the $k$ tensor indices. In particular $(R_{pq}^{(k)})^*=-R_{pq}^{(k)}$ and $R_{pq}^{(k)}$ commutes with $\Omega_{pq}$. Applying \eqref{eq:rotation_tensor_commutator} twice gives
    \begin{equation*}
        \nabla^k (\Omega_{pq}^2 f) =(\Omega_{pq}+R_{pq}^{(k)})^2 \nabla^k f.
    \end{equation*}
    Since $B$ is independent of $v$, we conclude by
    \begin{equation}\label{eq:Hk_evolution}
        (\partial_t-L) \nabla^k f=2k \nabla^k f+ \sum_{p<q} \left( 2R_{pq}^{(k)} \Omega_{pq}+ (R_{pq}^{(k)})^2 \right) \nabla^k f.
    \end{equation}

    We next compute the evolution of $\displaystyle \frac{|\nabla^k f|^2}{f}$. Write $L=a_{ij} \partial_{ij}+b_i \partial_i+6$, where
    \begin{equation*}
        a=B+ \sum_{p<q} (A_{pq}v) \otimes (A_{pq}v), \quad b=2v.
    \end{equation*}
    For an arbitrary tensor $H$, a direct differentiation of $\displaystyle\frac{|H|^2}{f}$ gives
    \begin{align*}
        (\partial_t-L)\frac{|H|^2}{f}=&\, \frac{2 H:(\partial_t-L) H}{f} -\frac{|H|^2 (\partial_t-L) f}{f^2}\\
        &\, -\frac{2}{f} a_{ij} \left( \partial_i H-H \partial_i \log f \right): \left( \partial_j H-H \partial_j \log f \right).
    \end{align*}
    For the specific choice of $H=\nabla^k f$, we insert \eqref{eq:Hk_evolution} in the first term and notice that the second term vanishes:
    \begin{align*}
        (\partial_t-L) \frac{|\nabla^k f|^2}{f}=&\, 4k \frac{|\nabla^k f|^2}{f}+ \frac{2}{f} \sum_{p<q} \nabla^k f : \left( 2R_{pq}^{(k)} \Omega_{pq}+(R_{pq}^{(k)})^2 \right) \nabla^k f \\
        &\, -\frac{2}{f} a_{ij} \left( \partial_i \nabla^k f-\nabla^k f \partial_i \log f \right): \left( \partial_j \nabla^k f-\nabla^k f \partial_j \log f \right).
    \end{align*}
    We split $a_{ij}$ in the last term. The part with $B(t)$ forms a negative square itself, and the remaining tensor part is combined with the second term together:
    \begin{align*}
        &\, -\frac{2}{f} \left| \Omega_{pq} \nabla^k f-\nabla^k f \cdot \Omega_{pq} \log f \right|^2+\frac{2}{f} \nabla^k f:\left( 2R_{pq}^{(k)} \Omega_{pq}+(R_{pq}^{(k)})^2 \right) \nabla^k f\\
        =&\, -\frac{2}{f} \left| \Omega_{pq} \nabla^k f+ R_{pq}^{(k)} \nabla^k f-\nabla^k f \cdot \Omega_{pq} \log f \right|^2.
    \end{align*}
    Here we use $(R_{pq}^{(k)})^\ast=-R_{pq}^{(k)}$ and in particular $\nabla^k f : R_{pq}^{(k)} \nabla^k f=0$. Therefore, we conclude by
    \begin{equation*}
        \begin{aligned}
            (\partial_t-L) \frac{|\nabla^k f|^2}{f}=4k \frac{|\nabla^k f|^2}{f} &\, -\frac{2}{f}\sum_{p<q} |\Omega_{pq} \nabla^k f+R_{pq}^{(k)} \nabla^k f-\nabla^k f \Omega_{pq} \log f|^2\\
            &\, -\frac{2}{f} B_{ij} (\partial_i \nabla^k f-\nabla^k f \partial_i \log f):(\partial_j \nabla^k f-\nabla^k f \partial_j \log f).
        \end{aligned}
    \end{equation*}
    In particular, we abandon the last two terms to get
    \begin{equation*}
        (\partial_t-L) \frac{|\nabla^k f|^2}{f} \leq 4k\frac{|\nabla^k f|^2}{f},
    \end{equation*}
    and abandon only the second term but keep the last term to get
    \begin{equation*}
        (\partial_t-L) \frac{|\nabla^k f|^2}{f} \leq 4k \frac{|\nabla^k f|^2}{f}-\eta \frac{|\nabla^{(k+1)} f|^2}{f}+ C|\nabla\log f|^2 \frac{|\nabla^k f|^2}{f},
    \end{equation*}
    where the second inequality follows from ellipticity $B(t) \geq \eta \Id$ and the Cauchy--Schwarz inequality.
\end{proof}

\begin{lem}[$\Gamma$-calculus of logarithms]\label{lem:logarithmic_weight}
    Denote by $F_k=f (A+\lambda t-\log f)^k$, where $\lambda>6$ and $A$ is sufficiently large such that $A-\log f \geq 2(m-1)$. For any $1 \leq k \leq m$, we have
    \begin{equation*}
        (\partial_t-L)F_k \geq k (\lambda-6) F_{k-1}+\frac{\eta}{2} k |\nabla \log f|^2 F_{k-1}.
    \end{equation*}
\end{lem}

\begin{proof}
    The starting point is to derive the equation solved by $\log f$. Recalling $L=a_{ij}\partial_{ij}+b_i\partial_i+6$, we compute directly
    \begin{align*}
        \partial_t \log f=a_{ij}\partial_{ij} \log f+b_i \partial_i \log f+a_{ij}\partial_i \log f \partial_j \log f+6.
    \end{align*}
    For any two functions $u$ and $v$, the product rule for $L$ reads
    \begin{equation}\label{eq: tLu}
        (\partial_t-L)(uv) =u (\partial_t-L)v+v (\partial_t-L)u -2a_{ij} \partial_i u \partial_jv+6uv.
    \end{equation}
    And hence for $Y=A+\lambda t-\log f$ we have
    \begin{align*}
        (\partial_t-L)(Y^k)=k Y^{k-1}\left( \lambda-6-a_{ij} \partial_i \log f \partial_j \log f \right)-k(k-1)Y^{k-2}a_{ij} \partial_i \log f \partial_j \log f-6Y^k.
    \end{align*}
    We apply \eqref{eq: tLu} with $u=f$ and $v=Y^k$ to obtain
    \begin{align*}
        (\partial_t-L)(fY^k)=&\, kfY^{k-1} \left( \lambda-6 -a_{ij} \partial_i \log f \partial_j \log f \right)\\
        &\, -k(k-1)f Y^{k-2} a_{ij} \partial_i \log f \partial_j \log f+2kf Y^{k-1} a_{ij}\partial_i \log f \partial_j \log f.
    \end{align*}
    Collecting the terms containing $f a_{ij} \partial_i \log f \partial_j \log f \geq \eta |\nabla \log f|^2f$ yields
    \begin{equation*}
        (\partial_t-L)(fY^k) \geq k(\lambda-6)f Y^{k-1}+kY^{k-2} (Y-k+1) \eta |\nabla \log f|^2f,
    \end{equation*}
    which yields the desired estimate by noticing $Y-k+1 \geq Y/2$ by construction.
\end{proof}

\begin{proof}[Proof of Theorem \ref{the:logarithmic_derivative}]
    With Lemma \ref{lem:higher_order_quotient} and Lemma \ref{lem:logarithmic_weight}, we carry out a series of Bernstein auxiliary functions to derive our logarithmic derivative estimates at all orders.
    
    We start with $k=1$, where we have
    \begin{equation*}
        (\partial_t-L) E_1 \leq 4E_1, \quad (\partial_t-L) F_1 \geq \frac{\eta}{2} |\nabla \log f|^2 f.
    \end{equation*}
    Since $E_1=|\nabla \log f|^2 f$, we get
    \begin{equation*}
        (\partial_t-L) (E_1-A_1F_1) \leq 0
    \end{equation*}
    as long as $A_1>8/\eta$. By the initial assumption on $\nabla \log f_0$ and Gaussian upper bound, we may choose $A_1$ large enough but only depending on $L_0$ and $L_1$ such that $E_1-A_1F_1 \leq 0$ holds initially. By the parabolic maximum principle (see \cite[Proposition 3.1]{carrillo2025relative}), we conclude by
    \begin{equation*}
        E_1 \leq A_1 F_1, \quad \forall \, t \geq 0.
    \end{equation*}

    Now we proceed by induction on $k$. Assume that we have $E_k \leq A_k F_k$ for any $t \geq 0$ for some $A_k$ depending on $L_0, \cdots, L_k$. Now we combine the following estimates
    \begin{equation*}
        (\partial_t-L) E_k \leq 4kE_k-\eta E_{k+1}+C|\nabla \log f|^2 E_k, \quad (\partial_t-L) E_{k+1} \leq 4(k+1) E_{k+1},
    \end{equation*}
    \begin{equation*}
        (\partial_t-L) F_{k+1} \geq (k+1)(\lambda-6)F_k+\frac{\eta}{2}(k+1) |\nabla \log f|^2 F_k.
    \end{equation*}
    Using the induction hypothesis, we get
    \begin{equation*}
        (\partial_t-L) \left(E_{k+1}+\frac{4(k+1)}{\eta}E_k-A_{k+1}F_{k+1}\right) \leq 0
    \end{equation*}
    as long as $\displaystyle A_{k+1} \geq \max\left( \frac{8C}{\eta^2}, \frac{16k}{\eta(\lambda-6)} \right) A_k$. By the initial assumption on $f_0$, we may also choose $A_{k+1}$ large enough and depending on $L_0, \cdots, L_{k+1}$ such that $E_{k+1}-A_{k+1}F_{k+1} \leq 0$ holds initially. By the parabolic maximum principle, we verify the estimate for $k+1$.

    Therefore, we have proved that
    \begin{equation*}
        \frac{|\nabla^k f|^2}{f}(t,v) \leq A_k f(t,v) (A+\lambda t-\log f(t,v))^k.
    \end{equation*}
    By the Gaussian lower bound, the right-hand side is bounded by $C_k(1+t+|v|^2)^k f(t,v)$ for some $C_k$. Divide by $f(t,v)$ and take the square root on both sides, and the result follows from elementary identities between $\nabla^k \log f$ and $\nabla^k f/f$.
\end{proof}

\section{Evolution of the relative entropy and preliminary results}\label{sec:entropy_evolution}

In this section, we first derive the evolution equation for the normalized relative entropy and then establish several preliminary moment estimates. 

\subsection{Evolution of the relative entropy}

This subsection is devoted to the time evolution formula of the relative entropy.

\begin{lem}[Evolution of relative entropy]\label{lem:evolution_entropy}
   For any $t \geq 0$, it holds that
   \begin{align*}
      H_N \left( F_N(t)| f_t^{\otimes N} \right) &\, \leq H_N \left( F_N(0)| f_0^{\otimes N} \right) \\
       &\, +\int_0^t \int_{\R^{3N}} F_N(s) \frac{1}{N} \sum_{i=1}^N \left[ \frac{1}{N} \sum_{j=1}^N V_{f}(v^i,v^j) -\bar V_{f}(v^i) \right]\ud s,
    \end{align*}
    where the two-body mean-field defect functional $V_f$ is defined by
    \begin{align*}
        V_{f}(v,w)= a(v-w): \frac{\nabla^2 f}{f}(v)&\, +2b(v-w) \cdot \nabla \log f(v)\\
        &\, -a(v-w):\nabla \log f(v) \otimes \nabla \log f(w),
    \end{align*}
    and $\bar V_f$ is the one-side expectation of $V_f$ under $f$:
    \begin{equation*}
      \bar  V_{f}(v)=\int_{\R^3} V_{f}(v,w)f(w) \ud w.
    \end{equation*}
\end{lem}

\begin{proof}
    We start by differentiating the normalized relative entropy directly.
    \begin{equation}\label{eq:H_derivative_start}
        \frac{\ud}{\ud t} H_N \left( F_N| f^{\otimes N} \right)= \frac{1}{N} \int (\partial_t F_N) \log \frac{F_N}{f^{\otimes N}}-\frac{1}{N} \int F_N \partial_t \log f^{\otimes N}.
    \end{equation}
    Recall the Landau master equation \eqref{eq:master-equation} solved by $F_N$:
    \begin{equation*}
        \partial_t F_N=\frac{1}{N} \sum_{i<j} (\nabla_{v^i}-\nabla_{v^j}) \cdot \big( a(v^i-v^j) \cdot (\nabla_{v^i}-\nabla_{v^j}) F_N \big).
    \end{equation*}
    We insert this formula into the first term of \eqref{eq:H_derivative_start} and integrate by parts.
    \begin{align}\label{eq:FN_entropy_part}
        \frac{1}{N} \int (\partial_t F_N) \log \frac{F_N}{f^{\otimes N}} &\, = -\frac{1}{N^2}\sum_{i<j} \int F_N \left[ a(v^i-v^j) : (\nabla_{v^i}-\nabla_{v^j}) \log F_N \otimes (\nabla_{v^i}-\nabla_{v^j}) \log \frac{F_N}{f^{\otimes N}} \right] \nonumber\\
        &\, = -\mathcal D_N(t) -\frac{1}{N^2}\sum_{i<j} \int F_N \left[ a(v^i-v^j): (\nabla_{v^i}-\nabla_{v^j})\log f^{\otimes N} \otimes (\nabla_{v^i}-\nabla_{v^j}) \log \frac{F_N}{f^{\otimes N}} \right],
    \end{align}
    where $\mathcal D_N(t)$ is the relative entropy dissipation functional defined by
    \begin{equation*}
        \mathcal D_N(t)= \frac{1}{N^2} \sum_{i<j} \int F_N \left|\sigma(v_i-v_j) \cdot (\nabla_{v^i}-\nabla_{v^j}) \log \frac{F_N}{f^{\otimes N}} \right|^2 \geq 0
    \end{equation*}
    with $\sigma(z)=|z|\Pi(z)$. We integrate by parts again in the last term in \eqref{eq:FN_entropy_part}.
    \begin{align}
        &\, -\frac{1}{N^2} \sum_{i<j} \int F_N \left[ a(v^i-v^j): (\nabla_{v^i}-\nabla_{v^j}) \log f^{\otimes N} \otimes (\nabla_{v^i}-\nabla_{v^j}) \log \frac{F_N}{f^{\otimes N}} \right] \nonumber\\
        =&\, \frac{1}{N^2} \sum_{i<j} \int F_N \left[ (\nabla_{v^i}-\nabla_{v^j}) \cdot \big(a(v^i-v^j) \cdot (\nabla_{v^i}-\nabla_{v^j}) \log f^{\otimes N} \big) +\left| \sigma(v_i-v_j) \cdot (\nabla_{v^i}-\nabla_{v^j}) \log f^{\otimes N} \right|^2 \right].
    \end{align}
    By the elementary identity $(\nabla_{v^i}-\nabla_{v^j})\log f^{\otimes N}=\nabla \log f(v_i)-\nabla \log f(v_j)$, and since $a$ is even and $b=\nabla\cdot a$ is odd, a direct expansion gives
    \begin{align*}
        &\, (\nabla_{v^i}-\nabla_{v^j}) \cdot \big( a(v^i-v^j) \cdot (\nabla \log f(v_i)-\nabla \log f(v_j)) \big)+ \left|\sigma(v_i-v_j) \cdot (\nabla \log f(v_i)-\nabla \log f(v_j)) \right|^2 \\
        =&\, V_{f}(v^i,v^j)+V_{f}(v^j,v^i).
    \end{align*}
    Hence, we sum up in $i,j$ and get
    \begin{equation}\label{eq:LN_G_over_G}
        -\frac{1}{N^2} \sum_{i<j} a(v^i-v^j): (\nabla_{v^i}-\nabla_{v^j}) \log f^{\otimes N} \otimes (\nabla_{v^i}-\nabla_{v^j}) \log \frac{F_N}{f^{\otimes N}}= \frac{1}{N^2}\sum_{i,j=1}^N V_{f}(v^i,v^j),
    \end{equation}
    where the diagonal terms vanish because we set $a(0)=0$ and $b(0)=0$, which implies finally
    \begin{align*}
        \frac{1}{N} \int (\partial_t F_N) \log \frac{F_N}{f^{\otimes N}}+\mathcal{D}_N(t)= \int  F_N(t) \left[ \frac{1}{N^2} \sum_{i,j=1}^N V_{f}(v^i,v^j) \right].
    \end{align*}
    It remains to identify the one-side expectation terms. Integrating by parts gives
    \begin{equation*}
        a \ast (f\nabla \log f)=b \ast f, \qquad \int (b(v-w) \cdot \nabla \log f(w)) f(w) \ud w=(c \ast f)(v).
    \end{equation*}
    Therefore, we compute explicitly
    \begin{align*}
        \int V_{f}(v,w) f(w) \ud w &\, =(a \ast f)(v):\frac{\nabla^2 f}{f}(v)+(b \ast f)(v) \cdot \nabla \log f(v)-(c \ast f)(v) \nonumber\\
        &\, -\nabla \log f(v) \cdot (a \ast (f\nabla \log f))(v) \nonumber\\
        &\, =(a \ast f)(v):\frac{\nabla^2 f}{f}(v)-(c \ast f)(v)=\partial_t \log f(v),
    \end{align*}
    where the last equality is the non-divergence form of the Landau equation. Consequently,
    \begin{equation*}
        \partial_t \log f^{\otimes N} =\sum_{i=1}^N \int V_{f}(v^i,w)f(w) \ud w.
    \end{equation*}
    Combining this identity with \eqref{eq:H_derivative_start}-\eqref{eq:LN_G_over_G} gives
    \begin{equation*}
        \frac{\ud}{\ud t} H_N \left( F_N| f^{\otimes N} \right) =-\mathcal D_N(t)+\int F_N(t) \left[ \frac{1}{N} \sum_{i=1}^N \left( \frac{1}{N} \sum_{j=1}^N V_{f}(v^i,v^j) -\bar V_{f}(v^i)\right) \right].
    \end{equation*}
    Since $\mathcal D_N(t)\geq 0$, the result follows from a time integration. 
\end{proof}

\subsection{Some preliminary results}

In this subsection, we gather some preliminary computations for Kac's system \eqref{eq:particle-sde}. The first lemma proves a uniform-in-time polynomial moment estimate for the particle system. The quantitative estimate \eqref{eq:conditional_moment_bound} will play a key role in the proof of Proposition \ref{prop:landau_exponential_moments}. 

\begin{lem}[Uniform-in-time polynomial moment]\label{lem:particle_moment_propagation}
    Let $k \geq 1$ and assume $M_{2k}<\infty$. Then there exists a constant $C_k>0$, independent of $N$ and $t$, such that
    \begin{equation}\label{eq:uniform_particle_moment}
        \sup_{N \geq 2} \sup_{t \geq 0} \E |V^1(t)|^{2k} \leq C_k M_{2k}.
    \end{equation}
    More precisely, conditioning on the initial configuration, we have
    \begin{equation}\label{eq:conditional_moment_bound}
        \sup_{N \geq 2} \sup_{t \geq 0} \E \left[ |V^1(t)|^{2k} \big| \bV(0) \right] \leq \max \left\{ |V^1(0)|^{2k}, (ke_N)^k \left( 1+\frac{k-1}{N} \right)^{-k} \right\},
    \end{equation}
    where we recall $\displaystyle e_N=\frac{1}{N} \sum_{i=1}^N |V^i(0)|^2=\frac{1}{N} \sum_{j=1}^N |V^j(t)|^2$ from \eqref{eq:particle-conservation}.
\end{lem}

\begin{proof}
    Fix the initial configuration and define
    \begin{equation*}
        m_k(t)=\E \left[ |V^1(t)|^{2k} \big| \bV(0) \right].
    \end{equation*}
    Applying It\^o's formula to $\phi(v)=|v|^{2k}$ and using the explicit computations
    \begin{equation*}
        \nabla \phi(v)=2k|v|^{2k-2} v, \quad \nabla^2 \phi(v)=2k|v|^{2k-2} \Id+2k(2k-2)|v|^{2k-4} v\otimes v,
    \end{equation*}
    together with $b(z)=-2z$ and $a(z)=|z|^2\Id-z \otimes z$, we have
    \begin{align*}
        \frac{\ud}{\ud t} m_k(t) =\frac{4k}{N} \sum_{j=1}^N \E \Big[ -|V^1|^{2k} &\, +k|V^1|^{2k-2} |V^j|^2 \nonumber\\
        &\, -(k-1)|V^1|^{2k-4} |V^1 \cdot V^j|^2 \Big| \bV(0) \Big].
    \end{align*}
    The first line can be expressed in an explicit way. For the second line, we only keep the non-positive term with $j=1$ and drop the remaining non-positive terms. Using \eqref{eq:particle-conservation} and Jensen's inequality, we write
    \begin{align*}
        \frac{\ud}{\ud t} m_k(t) &\, \leq -4k \left( 1+\frac{k-1}{N} \right) m_k(t) +4k^2 e_N \E \left[ |V^1(t)|^{2k-2} \big| \bV(0) \right]\\
        &\, \leq -4k \left( 1+\frac{k-1}{N} \right) m_k(t) +4k^2 e_N m_k(t)^{1-1/k}.
    \end{align*}
    We denote $u_k(t)=m_k(t)^{1/k}$. The preceding inequality gives
    \begin{equation*}
        u_k'(t) \leq -4 \left( 1+\frac{k-1}{N} \right) u_k(t)+ 4ke_N.
    \end{equation*}
    Solving this scalar differential inequality gives
    \begin{equation*}
        u_k(t) \leq \max \left\{ u_k(0), \left( 1+\frac{k-1}{N} \right)^{-1} k e_N \right\}.
    \end{equation*}
    Taking the $k$-th power on both sides yields \eqref{eq:conditional_moment_bound}. Finally, by convexity, we have
    \begin{equation*}
        \E [e_N^k] =\E \left( \frac{1}{N} \sum_{i=1}^N |V^i(0)|^2 \right)^k \leq \frac{1}{N} \sum_{i=1}^N \E |V^i(0)|^{2k}=M_{2k}.
    \end{equation*}
    Taking expectations in \eqref{eq:conditional_moment_bound} proves \eqref{eq:uniform_particle_moment}.
\end{proof}

The next lemma estimates the initial energy fluctuation.

\begin{lem}[Initial energy fluctuation]\label{lem:initial_energy_fluctuation}
    Assume $M_4<\infty$. Then we have
    \begin{equation*}
       \E \left( \frac{1}{N} \sum_{i=1}^N |V^i(0)|^2-3 \right)^2 \leq \frac{C}{N}(M_4+1).
    \end{equation*}
\end{lem}

\begin{proof}
    The variables $V^1(0),\cdots,V^N(0)$ are i.i.d. with $\E |V^1(0)|^2=3$. Hence all the cross terms vanish after expanding the square.
    \begin{equation*}
        \E \left( \frac{1}{N} \sum_{i=1}^N |V^i(0)|^2-3 \right)^2 = \frac{1}{N} \E \left( |V^1(0)|^2-3 \right)^2 \leq \frac{C}{N}(M_4+1).
    \end{equation*}
\end{proof}

Finally we compute an empirical covariance fluctuation estimate.

\begin{lem}[Empirical covariance fluctuation]\label{lem:entrywise_covariance_fluctuation}
    Assume $M_{16}<\infty$. For every $\alpha,\beta\in\{1,2,3\}$,
    \begin{equation}\label{eq:entrywise_covariance_fluctuation}
        \sup_{t \geq 0} \E \left| \frac{1}{N} \sum_{j=1}^N V^j_\alpha(t) V^j_\beta(t)-E_{\alpha\beta}(t) \right|^2 \leq \frac{C}{N}(M_4+1),
    \end{equation}
    \begin{align}\label{eq:m16}
        \sup_{t \geq 0} \E \left| \frac{1}{N} \sum_{j=1}^N V^j_\alpha(t) V^j_\beta(t)-E_{\alpha\beta}(t) \right|^8 \leq \frac{C}{N^4}(M_{16}+1).
    \end{align}
    Here $E_{\alpha\beta}(t)$ is the covariance matrix of the Landau equation.
\end{lem}

\begin{proof}
    We define the shorthand notation
    \begin{equation*}
        Q_{\alpha\beta}(t)= \frac{1}{N} \sum_{j=1}^N V^j_\alpha(t) V^j_\beta(t).
    \end{equation*}
    We also recall
    \begin{equation*}
        m_N=\frac{1}{N} \sum_{i=1}^N V^i(0)=\frac{1}{N} \sum_{i=1}^N V^i(t), \quad e_N=\frac{1}{N} \sum_{i=1}^N |V^i(0)|^2=\frac{1}{N} \sum_{i=1}^N |V^i(t)|^2
    \end{equation*}
    from \eqref{eq:particle-conservation}. Applying It\^o's formula for the particle system \eqref{eq:particle-sde} gives
    \begin{equation*}
        \ud Q_{\alpha\beta}(t)=\left[ -12Q_{\alpha\beta}(t)+12 m_{N,\alpha} m_{N,\beta} +4\delta_{\alpha\beta} \left( e_N-|m_N|^2 \right) \right] \ud t+\ud M_{\alpha\beta}(t),
    \end{equation*}
    where the martingale $M_{\alpha\beta}$ is given by
    \begin{equation*}
        \ud M_{\alpha\beta}(t)= \frac{\sqrt{2}}{N^{3/2}} \sum_{i<j} \sum_{\gamma=1}^3 \left[ (V^i_\beta-V^j_\beta) \sigma_{\alpha\gamma}(V^i-V^j)+ (V^i_\alpha-V^j_\alpha)
        \sigma_{\beta\gamma}(V^i-V^j) \right] \ud B_t^{i,j,\gamma}.
    \end{equation*}
    Since $|\sigma(z)|\leq C|z|$, its quadratic variation satisfies
    \begin{equation*}
        \frac{\ud}{\ud t} \langle M_{\alpha\beta} \rangle_t \leq \frac{C}{N^3} \sum_{i<j} |V^i(t)-V^j(t)|^4 \leq \frac{C}{N^2} \sum_{i=1}^N |V^i(t)|^4,
    \end{equation*}
    where we use $|V^i-V^j|^4\leq 8(|V^i|^4+|V^j|^4)$ in the last step.
    The covariance matrix of the Landau solution satisfies
    \begin{equation*}
        E_{\alpha\beta}'(t)=-12E_{\alpha\beta}(t)+12\delta_{\alpha\beta},
    \end{equation*}
    see for instance in \cite{villani1998spatially}. Let $\displaystyle D_{\alpha\beta}(t)=Q_{\alpha\beta}(t)-E_{\alpha\beta}(t)$ be the target difference. Subtracting the preceding two equations and solving the linear equation, we have
    \begin{equation}\label{eq:entrywise_D_solution_revised}
        D_{\alpha\beta}(t)=e^{-12t} D_{\alpha\beta}(0)+(1-e^{-12t}) R_{\alpha\beta}+ \int_0^t e^{-12(t-s)} \ud M_{\alpha\beta}(s),
    \end{equation}
    where $R_{\alpha\beta}$ is a constant defined by
    \begin{equation*}
        R_{\alpha\beta}=
        \begin{cases}
            m_{N,\alpha} m_{N,\beta}, &\, \alpha \neq \beta;\\
            m_{N,\alpha}^2+\frac{1}{3} (e_N-|m_N|^2-3), \quad &\, \alpha=\beta.
        \end{cases}
    \end{equation*}
    To prove \eqref{eq:entrywise_covariance_fluctuation}, we estimate the three terms in \eqref{eq:entrywise_D_solution_revised} separately. By independence at $t=0$, in a similar spirit as Lemma \ref{lem:initial_energy_fluctuation} we have
    \begin{equation*}
        \E |D_{\alpha\beta}(0)|^2 \leq \frac{C}{N}(M_4+1), \quad \E |m_N|^4 \leq \frac{C}{N^2} (M_4+1).
    \end{equation*}
    Together with Lemma \ref{lem:initial_energy_fluctuation}, this gives
    \begin{equation*}
        \E |R_{\alpha\beta}|^2 \leq \frac{C}{N}(M_4+1).
    \end{equation*}
    Finally, for the martingale term, by It\^o's isometry and the quadratic-variation bound, we have
    \begin{align*}
        \E \left| \int_0^t e^{-12(t-s)} \ud M_{\alpha\beta}(s) \right|^2 \leq \frac{C}{N^2}\int_0^t e^{-24(t-s)} \E \sum_{i=1}^N |V^i(s)|^4 \ud s.
    \end{align*}
    Applying Lemma \ref{lem:particle_moment_propagation} with $k=2$ gives the desired bound \eqref{eq:entrywise_covariance_fluctuation}. As for \eqref{eq:m16}, we also start with \eqref{eq:entrywise_D_solution_revised}. Independence and the standard eighth-moment estimate for centered empirical averages give
    \begin{equation*}
        \E |D_{\alpha\beta}(0)|^8 \leq \frac{C}{N^4} (M_{16}+1), \quad \E |e_N-3|^8 \leq \frac{C}{N^4} (M_{16}+1), \quad \E |m_N|^{16} \leq \frac{C}{N^8} (M_{16}+1).
    \end{equation*}
    Consequently, the first two terms in \eqref{eq:entrywise_D_solution_revised} satisfy
    \begin{equation*}
        \E \left| e^{-12t} D_{\alpha\beta}(0)+(1-e^{-12t}) R_{\alpha\beta} \right|^8
        \leq \frac{C}{N^4} (M_{16}+1).
    \end{equation*}
    For the martingale term, we apply the Burkholder--Davis--Gundy inequality and Lemma \ref{lem:particle_moment_propagation} with $k=8$ to show
    \begin{align*}
        \E \left( \int_0^t e^{-12(t-s)} \ud M_{\alpha\beta}(s) \right)^8 
        &\, \leq C\E \left( \int_0^t e^{-24(t-s)}
        \ud\langle M_{\alpha\beta}\rangle_s \right)^4 \\
        &\, \leq \frac{C}{N^8}\E \left( \int_0^t e^{-24(t-s)}
        \sum_{i=1}^N |V^i(s)|^4\ud s \right)^4 \\
        &\, \leq \frac{C}{N^4}(M_{16}+1).
    \end{align*}
    This completes the proof of \eqref{eq:m16}.
\end{proof}

\section{Law-of-large-numbers estimates}\label{sec:lln}

This section is devoted to the proof of a new law-of-large-numbers estimate involving two logarithmic derivatives. 

\begin{lem}[Dual backward solutions]\label{lem:polynomial_backward_estimates}
    Fix any terminal time $T>0$ and any test function $\psi: \R^3 \to \R$ satisfying the pointwise polynomial estimate
    \begin{equation}\label{eq:psi_polynomial_bound}
        |\nabla^r \psi(v)| \leq A_T (1+|v|^{n_r})
    \end{equation}
    for $0 \leq r \leq m$ for some constants $A_T>0$, $m \geq 0$ and $n_0, \cdots, n_m \geq 0$. Then, for every $0 \leq t \leq T$, there exists a family of functions $u_s: \R^3 \to \R$ with $0 \leq s \leq t$, satisfying the dual backward equation
    \begin{equation}\label{eq:backward_linear_landau}
    \begin{aligned}
        \partial_s u_s &\, = -(a \ast f_s): \nabla^2 u_s -2 (b \ast f_s) \cdot \nabla u_s \\
        &\, = -\nabla \cdot \left( (a \ast f_s) \cdot \nabla u_s +(b \ast f_s) u_s \right) +(c \ast f_s)u_s,
    \end{aligned}
    \end{equation}
    with terminal condition $u_t=\psi$. Moreover, we have the conservation condition
    \begin{equation}\label{eq:conservation_condition}
        \int_{\R^3} \psi(v) f_t(v) dv =\int_{\R^3} u_0(v) f_0(v) dv,
    \end{equation}
    and the pointwise polynomial estimate
    \begin{equation}\label{eq:backward_growth_bound}
        |\nabla^r u_s(v)| \leq CA_T (1+|v|^{n_r}), \quad  0 \leq s \leq t \leq T
    \end{equation}
    for $0 \leq r \leq m$, where $C>0$ is a universal constant.
\end{lem}

\begin{rem}
    This lemma could be compared to  the duality method of Bresch--Duerinckx--Jabin \cite{bresch2024duality}. But our method works on the level of (the limit) Landau equation,  instead of  the dual master equation.
    Any test function $\psi$ acting on $f$ at time $t$ can be pulled back to time zero via the adjoint of the limiting semigroup, while \eqref{eq:backward_growth_bound} ensures that this dual evolution preserves the same polynomial class for the test function.
\end{rem}

\begin{proof}
    {\bf Step 1: Construction of the solution.} We consider the linear stochastic equation
    \begin{equation}\label{eq:backward_sde}
        \ud X_r=-4X_r \ud r +\sqrt{2} \sum_{k=1}^3 (e_k\times X_r) \ud W_r^k +\sqrt{2(3\Id-E(r))} \ud B_r, \quad X_s=v,
    \end{equation}
    where $(e_k)_{k=1}^3$ denotes the canonical basis of $\R^3$, and $(W^k)_{k=1}^3$ are independent standard one-dimensional Brownian motions, and $B$ is a standard three-dimensional Brownian motion independent of $(W^k)_{k=1}^3$. Since the coefficients are affine in the state variable, and the diffusive matrix $3\Id-E(r)$ is non-negative definite and uniformly bounded under the energy normalization condition \eqref{eq:normalization-condition}, \eqref{eq:backward_sde} admits a pathwise unique strong solution. 

    For any $0 \leq s \leq t \leq T$ and any $v \in \R^3$, let $X_r^{s,v}$ denote the solution of \eqref{eq:backward_sde} starting from $v$ at time $s$, and define
    \begin{equation*}
        u_s(v)= \E \left[ \psi(X_t^{s,v}) \right].
    \end{equation*}
    In the following steps, we show that $u$ satisfies all the properties stated in the lemma. The terminal condition follows immediately from the definition:
    \begin{align*}
        u_t(v)= \E \left[ \psi(X_t^{t,v}) \right]=\psi(v).
    \end{align*}

    {\bf Step 2: Verification of the pointwise estimate.} We next verify \eqref{eq:backward_growth_bound}. Since the stochastic flow associated with \eqref{eq:backward_sde} depends affinely on the initial condition $v$, its derivative $D_vX_r^{s,v}$ is independent of $v$. For any $x\in\R^3$,  the process $(D_vX_r^{s,v}(x))_{s\le r\le t}$ satisfies $D_v X_s^{s,v}(x)=x$ since $ X_s^{s,v}=v$, and we compute
    \begin{align*}
        \ud D_v X_r^{s,v} (x)= -4 D_v X_r^{s,v} (x) \ud r +\sqrt{2} \sum_{k=1}^3 [e_k \times D_v X_r^{s,v} (x)] \ud W_r^k.
    \end{align*}
    Since $\displaystyle \sum_{k=1}^3 |e_k\times D_v X_r^{s,v} (x)|^2=2| D_v X_r^{s,v} (x)|^2$, by It\^o's formula we have
    \begin{align*}
        \ud | D_v X_r^{s,v} (x)|^2=&\, 2D_v X_r^{s,v} (x) \cdot \ud  D_v X_r^{s,v} (x)+2 \sum_{k=1}^3 |e_k\times D_v X_r^{s,v} (x)|^2 \ud r\\
        =&\, -4|D_v X_r^{s,v} (x)|^2 \ud r +\sqrt{2} \sum_{k=1}^3 \underbrace{D_v X_r^{s,v} (x)\cdot [e_k\times D_v X_r^{s,v} (x)]}_{=0} \ud W_r^k\\
        =&\, -4|D_v X_r^{s,v} (x)|^2 \ud r,
    \end{align*}
    which in turn implies $|D_v X_r^{s,v}(x)|=e^{-2(r-s)}|x| \leq |x|$ for any $x \in \R^3$. In particular, the operator norm of $D_vX_r^{s,v}$ is bounded by $1$.

    We next derive the moment estimate for $X_r^{s,v}$. Applying It\^o's formula to $|X_r^{s,v}|^{2q}$ for any $q \geq 1$ and using the boundedness of $3\Id-E(r)$ gives
    \begin{align*}
        \frac{\ud}{\ud r}\E|X_r^{s,v}|^{2q} \leq -c_q \E |X_r^{s,v}|^{2q} +C_q \E |X_r^{s,v}|^{2q-2} \leq -\frac{c_q}{2} \E |X_r^{s,v}|^{2q} +C_q,
    \end{align*}
    where the last inequality follows from Young's inequality. Solving this scalar differential inequality gives the uniform moment estimate
    \begin{equation}\label{eq:backward_process_moments}
        \sup_{0 \leq s \leq r \leq t} \E[|X_r^{s,v}|^{2q}] \leq C_q(1+|v|^{2q})
    \end{equation}
    for some constant $C_q>0$ for any $q \geq 1$.

    We now differentiate the stochastic representation $u_s(v)=\E \left[ \psi(X_t^{s,v}) \right]$. Since the flow is affine in $v$, its second derivative with respect to $v$ vanishes. Therefore, differentiating under the expectation gives
    \begin{equation*}
        \nabla u_s(v)=\E \left[ (D_v X_t^{s,v})^{\mathsf T} \nabla \psi(X_t^{s,v}) \right],
    \end{equation*}
    and also
    \begin{equation*}
        \nabla^2 u_s(v)=\E \left[ (D_v X_t^{s,v})^{\mathsf T} \nabla^2 \psi(X_t^{s,v}) D_v X_t^{s,v} \right].
    \end{equation*}
    Combining these two identities with the upper bound of $D_v X_r^{s,v}$, the pointwise assumption on $\psi$ and the moment estimate \eqref{eq:backward_process_moments}, we get \eqref{eq:backward_growth_bound}.

    {\bf Step 3: Verification of the weak solution and the conservation condition.} It remains to verify that $u$ satisfies the dual backward equation \eqref{eq:backward_linear_landau} and the conservation condition \eqref{eq:conservation_condition}. First, using the elementary identity $\displaystyle \sum_{k=1}^3 (e_k\times v)(e_k\times v)^{\mathsf T}=|v|^2\Id- v \otimes v$, we obtain the quadratic variation matrix
    \begin{align*}
        \ud \langle X \rangle_r &\, = 2\sum_{k=1}^3 (e_k\times X_r)(e_k\times X_r)^\top \ud r +2(3\Id-E(r)) \,\ud r\\
        &\, = 2\left( |X_r|^2 \Id-X_r \otimes X_r+3 \Id-E(r) \right) \ud r.
    \end{align*}
    Consequently, the infinitesimal generator of $X$ is given, for a smooth function $\phi$, by
    \begin{align*}
        \mathcal L_r \phi(v) &\, =\left( |v|^2\Id-v \otimes v+3\Id-E(r) \right):\nabla^2 \phi(v)-4v \cdot \nabla \phi(v) \notag\\
        &\, =(a \ast f_r)(v):\nabla^2 \phi(v)+2(b \ast f_r)(v) \cdot \nabla \phi(v).
    \end{align*}
    By the flow property and the Markov property of $X$, we write
    \begin{equation*}
        u_s(v) =\E \left[ \psi(X_t^{s,v}) \right] =\E \left[ \E \left[ \psi(X_t^{s+h, X_{s+h}^{s,v}})| \mathcal F_{s+h} \right] \right] =\E \left[ u_{s+h}(X_{s+h}^{s,v}) \right].
    \end{equation*}
    Now we apply the time-dependent It\^o's formula to $u_{s+h}(X_{s+h}^{s,v})$ for $h \geq 0$. 
    \begin{align*}
        u_{s+h}(X_{s+h}^{s,v}) =u_s(v)&\, +\int_s^{s+h} (\partial_r+\mathcal L_r) u_r (X_r^{s,v}) \ud r+\sqrt{2} \sum_{k=1}^3 \int_s^{s+h} \nabla u_r (X_r^{s,v}) \cdot (e_k \times X_r^{s,v}) \ud W_r^k\\
        &\, +\int_s^{s+h} \nabla u_r (X_r^{s,v})^\top \sqrt{2(3\Id-E(r))} \ud B_r.
    \end{align*}
    Taking expectations on both sides, we get
    \begin{equation*}
        0=\E \left[ u_{s+h}(X_{s+h}^{s,v})-u_s(v) \right]=\int_s^{s+h} \E \left[ (\partial_r u_r+\mathcal L_r u_r)(X_r^{s,v}) \right] \ud r.
    \end{equation*}
    Since $X_r^{s,v} \to v$ as $r \to s$, the continuity of the coefficients and the estimates established in Step 2 allow us to take the limit $h \to 0$. We obtain
    \begin{equation*}
        0= \lim_{h \to 0} h^{-1} \int_s^{s+h} \E \left[ (\partial_r u+\mathcal L_r u) \right] \ud r=\partial_s u_s+\mathcal L_s u_s.
    \end{equation*}
    Thus $u$ satisfies \eqref{eq:backward_linear_landau}. It remains to prove \eqref{eq:conservation_condition}, which follows from a direct computation. By \eqref{eq:backward_linear_landau} and \eqref{eq: div-landau-equation} and integration by parts, we have
    \begin{align*}
        \frac{\ud}{\ud s} \int_{\R^3} u_s f_s \ud v = &\, \int_{\R^3} \partial_s u_s f_s \ud v+\int_{\R^3} u_s \partial_s f_s \ud v\\
        =&\, \int_{\R^3} \nabla f_s \cdot \left( (a \ast f_s) \nabla u_s+(b \ast f_s) u_s \right)+ (c \ast f_s) u_s f_s \ud v\\
        &\, \qquad \qquad \qquad \qquad -\int_{\R^3} \nabla u_s \cdot \left( (a \ast f_s) \nabla f_s-(b \ast f_s) f_s \right) \ud v =0. 
    \end{align*}
    Hence the integral $\displaystyle \int_{\R^3} u_s f_s$ is conserved in time, which yields \eqref{eq:conservation_condition} by comparing the values of $s=0$ and $s=t$.
\end{proof}

Using this dual backward solution, we prove a new law of large numbers for weighted moments, which cannot be treated directly as in subsection 3.2 using pointwise conservation laws.

\begin{lem}[LLN for weighted moments]\label{lem:score_weighted_moment_lln}
    For any $T>0$ and any weight function $\psi\in\mathcal{C}^2(\R^3)$ satisfying the assumptions of Lemma \ref{lem:polynomial_backward_estimates}, there exists a universal constant $C>0$ such that for any $0 \leq t \leq T$,
    \begin{equation*}
        \E \left[ \left( 1+\frac{1}{N} \sum_{i=1}^N |V^i(t)|^4 \right) \left| \frac{1}{N} \sum_{j=1}^N \psi(V^j(t)) -\int_{\R^3} \psi(v) f_t(v) \ud v \right|^2 \right] \leq \frac{C}{N} (1+T)^2A_T^2 M_n.
    \end{equation*}
    Here $n=8\max(n_0/2,n_1,n_2+4)$.
\end{lem}

\begin{proof}
    For simplicity, we use the following shorthand notation for the fluctuation 
    \begin{align*}
        R_\psi^N(t)=\frac{1}{N} \sum_{j=1}^N \psi(V^j(t))-\int_{\R^3} \psi(v) f_t(v) \ud v.
    \end{align*}

    {\bf Step 1: Removing the polynomial term.} By the Cauchy--Schwarz inequality, we have
    \begin{align*}
        \E \left[ \left( 1+\frac{1}{N} \sum_{i=1}^N |V^i(t)|^4 \right) \left| R_\psi^N(t) \right|^2 \right] \leq \left[ \E \left( 1+\frac{1}{N} \sum_{i=1}^N |V^i(t)|^4 \right)^2 \right]^{1/2} \left[ \E \left| R_\psi^N(t) \right|^4 \right]^{1/2}.
    \end{align*}
    Now the polynomial terms and the weight terms are separated. Applying Lemma \ref{lem:particle_moment_propagation} with $k=4$, we have
    \begin{equation}\label{eq:polynomial-term}
        \E \left( 1+\frac{1}{N} \sum_{i=1}^N |V^i(t)|^4 \right)^2 \leq 2+\frac{2}{N} \sum_{i=1}^N \E |V^i(t)|^8 \leq C(M_8+1).
    \end{equation}
    Therefore, it remains to prove the bound for $\E \left| R_\psi^N(t) \right|^4$.

    {\bf Step 2: Rewriting the weight term by duality.}  Fix $0 \leq t \leq T$, and let $(u_s)_{0 \leq s \leq t}$ be the solution of the dual backward equation constructed in Lemma \ref{lem:polynomial_backward_estimates}, with terminal condition $u_t=\psi$. We define the empirical measure
    \begin{equation*}
        \mu_N(s)= \frac{1}{N} \sum_{i=1}^N \delta_{V^i(s)}.
    \end{equation*}
    The guiding idea is to use the conservation condition \eqref{eq:conservation_condition} and Newton--Leibniz formula to rewrite
    \begin{equation*}
        R_\psi^N(t)= \langle u_t, \mu_N(t)-f_t \rangle=\langle u_0, \mu_N(0)-f_0 \rangle+\frac{1}{N} \sum_{i=1}^N \int_0^t \ud u_s \left( V^i(s) \right).
    \end{equation*}
    The first term is expected to be small since it is an initial fluctuation, and the second term carries the natural $O(1/N)$ order.

    We first rewrite the integral part. Applying the time-dependent It\^o's formula to $u_s \left( V^i(s) \right)$ and using the dual backward equation \eqref{eq:backward_linear_landau}, we have
    \begin{align}\label{eq:ito-phiv}
        \ud u_s \left( V^i(s) \right) = \Bigg[ \partial_s u_s \left( V^i(s) \right) &\, +\frac{2}{N} \sum_{j=1}^N b \left( V^i(s)-V^j(s) \right) \cdot \nabla u_s \left( V^i(s) \right) \nonumber\\
        &\, +\frac{1}{N} \sum_{j=1}^N a \left( V^i(s)-V^j(s) \right): \nabla^2 u_s \left( V^i(s) \right) \Bigg] \ud s \nonumber\\
        &\, + \frac{\sqrt{2}}{\sqrt{N}} \sum_{j=1}^N \nabla u_s \left( V^i(s) \right)^\top \sigma \left( V^i(s)-V^j(s) \right) \ud B_s^{i,j} \nonumber\\
        = &\, \left[ \left( a \ast \mu_N(s)-a \ast f_s \right):\nabla^2 u_s +2 \left( b \ast \mu_N(s)-b \ast f_s \right) \cdot \nabla u_s \right] \left( V^i(s) \right) \ud s +\ud M_s^i,
    \end{align}
    where the martingale term is given by
    \begin{equation*}
        M_t^i=\frac{\sqrt{2}}{\sqrt{N}} \sum_{j=1}^N \int_0^t \nabla u_s(V^i(s))^{\mathsf T} \sigma \left( V^i(s)-V^j(s) \right) \ud B_s^{i,j}.
    \end{equation*}
    Therefore, we substitute \eqref{eq:ito-phiv} into the formula of $R_\psi^N$ and get
    \begin{equation}\label{eq:empirical_backward_decomposition}
        R_\psi^N(t)= \langle u_0, \mu_N(0)-f_0 \rangle +\int_0^t \left\langle \mu_N(s), (a \ast \mu_N(s) -a \ast f_s):\nabla^2 u_s +2(b \ast \mu_N(s)-b \ast f_s) \cdot \nabla u_s \right\rangle \ud s+\mathcal M_t,
    \end{equation}
    where $\displaystyle \mathcal M_t=\frac{1}{N} \sum_{i=1}^N M_t^i$ is a martingale.
    
    It remains to estimate the three terms on the right-hand side. We begin with the initial fluctuation. Since the variables $(V^i(0))_{i=1}^N$ are i.i.d., the standard moment estimates for sums of centered independent random variables (similar to Lemma \ref{lem:initial_energy_fluctuation}) give
    \begin{equation}\label{eq:initial_fluctuation_fourth_moment}
        \E \left| \frac{1}{N} \sum_{i=1}^N u_0(V^i(0)) -\int_{\R^3} u_0 f_0 \right|^4 \leq \frac{C}{N^2} A_T^4 (M_{4n_0}+1).
    \end{equation}
    Here we used the polynomial growth estimate for $u_0$ from Lemma \ref{lem:polynomial_backward_estimates} and the corresponding moment assumption on $f_0$.

    {\bf Step 3: Estimating the martingale term.} We next estimate the fourth moment of the martingale term $\mathcal M_t$ in \eqref{eq:empirical_backward_decomposition}. Using $B^{j,i}=-B^{i,j}$ and $\sigma(-z)=\sigma(z)$, we symmetrize it as
    \begin{equation*}
        \mathcal M_t= \frac{1}{\sqrt{2N^3}} \sum_{i,j=1}^N \int_0^t \left[ \nabla u_s \left( V^i(s)\right)-\nabla u_s \left( V^j(s) \right) \right]^\top \sigma \left( V^i(s)-V^j(s) \right) \ud B_s^{i,j}.
    \end{equation*}
    Since $|\sigma(z)| \leq |z|$, the quadratic variation of $\mathcal M$ satisfies
    \begin{equation}\label{eq:martingale_quadratic_variation}
        \langle \mathcal M \rangle_t \leq \frac{C}{N^3} \sum_{i,j=1}^N \int_0^t |V^i(s)-V^j(s)|^2 \left| \nabla u_s(V^i(s))-\nabla u_s(V^j(s)) \right|^2 \ud s.
    \end{equation}
    The Hessian bound proved in Lemma \ref{lem:polynomial_backward_estimates} gives, by the mean-value theorem,
    \begin{equation*}
        \left| \nabla u_s(v)-\nabla u_s(w) \right| \leq CA_T |v-w| (1+|v|^{n_2}+|w|^{n_2}).
    \end{equation*}
    Substituting this estimate into \eqref{eq:martingale_quadratic_variation} yields
    \begin{equation*}
        \langle \mathcal M \rangle_t \leq \frac{C}{N} A_T^2 \int_0^t \left( 1+\frac{1}{N} \sum_{i=1}^N |V^i(s)|^{4+2n_2} \right) \ud s.
    \end{equation*}
    The square of the last empirical moment is uniformly integrable by Lemma \ref{lem:particle_moment_propagation}. Hence the Burkholder--Davis--Gundy inequality gives
    \begin{equation}\label{eq:martingale_fourth_moment}
        \E \left[ |\mathcal M_t|^4 \right] \leq C \E \left[ \langle \mathcal M \rangle^2_t \right] \leq \frac{C}{N^2} T^2 A_T^4 (M_{8+4n_2}+1).
    \end{equation}

    {\bf Step 4: Estimating the drift term.} It remains to estimate the drift term in \eqref{eq:empirical_backward_decomposition}. Recall the normalized empirical momentum, energy, and covariance matrix are defined by
    \begin{align*}
        m_N =\frac{1}{N} \sum_{i=1}^N V^i(s), \quad e_N =\frac{1}{N} \sum_{i=1}^N |V^i(s)|^2, \quad Q(s) =\frac{1}{N}\sum_{i=1}^N V^i(s) \otimes V^i(s).
    \end{align*}
    Since $f_s$ satisfies the normalization condition \eqref{eq:normalization-condition}, a direct computation gives
    \begin{align*}
        b \ast \mu_N(s) (v)-b \ast f_s(v) &\, = 2m_N,\\
        a \ast \mu_N(s) (v)-a \ast f_s(v) &\, = -2(v \cdot m_N) \Id +v \otimes m_N+ m_N \otimes v +(e_N-3) \Id -\left( Q(s)-E(s) \right).
    \end{align*}
    We first consider the part involving $b$ in the drift term, which writes 
    \begin{align*}
        \left\langle \mu_N(s), (b \ast \mu_N(s)-b \ast f_s) \cdot \nabla u_s \right\rangle= \frac{2}{N} m_N \cdot \sum_{i=1}^N \nabla u_s \left( V^i(s) \right).
    \end{align*}
    Hence, by the Cauchy--Schwarz inequality, the fourth moment is bounded by
    \begin{align}\label{eq:drift_b_fourth_moment}
        \E \left| \left\langle \mu_N(s), \left( b \ast \mu_N(s)(v)-b \ast f_s(v) \right)\cdot \nabla u_s \right\rangle \right|^4 \leq &\, \frac{C}{N^4} \E \left[ |m_N|^4 \left| \sum_{i=1}^N \nabla u_s (V^i(s)) \right|^4 \right] \nonumber \\
        \leq &\, \frac{C}{N^4} A_T^4 \underbrace{\left[ \E |m_N|^8 \right]^{1/2}}_{N^{-2}}\underbrace{\left[ \E  \left| \sum_{i=1}^N (1+|V^i(s)|^{n_1}) \right|^8 \right]^{1/2}}_{N^4} \nonumber \\
        \leq &\, \frac{C}{N^2} A_T^4 (M_{8n_1}+1).
    \end{align}
    Next for the part involving $a$ in the drift term, we write
    \begin{align*}
        &\, \left\langle \mu_N(s), (a \ast \mu_N(s)-a \ast f_s):\nabla^2 u_s \right\rangle\\
        &\, \quad= \frac{1}{N} \sum_{i=1}^N \Big[ -2 \left( V^i(s) \cdot m_N \right) \Id +V^i(s)\otimes m_N +m_N \otimes V^i(s)\\
        &\, \hspace{4cm} + (e_N-3) \Id -\left( Q(s)-E(s) \right) \Big] :\nabla^2 u_s \left( V^i(s)\right).
    \end{align*}
    Therefore, the fourth moment is bounded by the sum of the fourth moments of each term,
    \begin{align*}
        \E \left[ \left| \left\langle \mu_N(s), \left( a \ast \mu_N(s)(v)-a \ast f_s(v) \right):\nabla^2 u_s \right\rangle \right|^4 \right] \leq \frac{C}{N^4} \Bigg( \E &\, \left[ |m_N|^4 \left| \sum_{i=1}^N |V^i(s)|^4 |\nabla^2 u_s(V^i(s))| \right|^4 \right]\\
        +\E \left[ |e_N-3|^4 \left| \sum_{i=1}^N \nabla^2 u_s(V^i(s)) \right|^4 \right]+\E &\, \left[ |Q(s)-E(s)|^4 \left| \sum_{i=1}^N \nabla^2 u_s(V^i(s)) \right|^4 \right] \Bigg).
    \end{align*}
    The first two terms can be controlled directly by 
    \begin{align*}
        &\, \frac{C}{N^4} \left( \E \left[ |m_N|^4 \left| \sum_{i=1}^N |V^i(s)|^4 |\nabla^2 u_s(V^i(s))| \right|^4 \right] +\E \left[ |e_N-3|^4 \left| \sum_{i=1}^N \nabla^2 u_s(V^i(s)) \right|^4 \right] \right)\\
        \leq &\, \frac{C}{N^4} A_T^4 \left( \left[ \E |m_N|^8 \right]^{1/2} + \left[ \E |e_N-3|^8 \right]^{1/2} \right) \cdot \left[ \E \left| \sum_{i=1}^N (1+|V^i(s)|^{n_2+4}) \right|^8 \right]^{1/2}\\
        \leq &\, \frac{C}{N^2} A_T^4 M_{8n_2+32}.
    \end{align*}
    For the covariance term, by \eqref{eq:m16}, we get
    \begin{align*}
        \frac{C}{N^4} \E \left[ |Q(s)-E(s)|^4 \left| \sum_{i=1}^N \nabla^2 u_s(V^i(s)) \right|^4 \right] \leq &\, \frac{C}{N^4} \left[ \E |Q(s)-E(s)|^8 \right]^{1/2} \left[ \E \left| \sum_{i=1}^N \nabla^2 u_s(V^i(s)) \right|^8 \right]^{1/2}\\
        \leq &\, \frac{C}{N^6} A_T^4 \left[ \E \left| \sum_{i=1}^N (1+|V^i(s)|^{n_2}) \right|^8 \right]^{1/2}\\
        \leq &\, \frac{C}{N^2} A_T^4 (M_{8n_2}+1).
    \end{align*}
    Therefore, we get
    \begin{equation}\label{eq:weighted_coefficient_fourth_moment}
       \E \left| \left\langle \mu_N(s), \left( a \ast \mu_N(s)(v)-a \ast f_s(v) \right): \nabla^2 u_s \right\rangle \right|^4 \leq \frac{C}{N^2} A_T^4 (M_{8n_2+32}+1).
    \end{equation}
    We finish the proof by combining \eqref{eq:polynomial-term} \eqref{eq:initial_fluctuation_fourth_moment} \eqref{eq:martingale_fourth_moment} \eqref{eq:drift_b_fourth_moment} and \eqref{eq:weighted_coefficient_fourth_moment}.
\end{proof}

Now we apply the new law of large numbers to the Landau--Maxwellian setting.

\begin{prop}[Weighted LLN estimate]\label{prop:weighted_lln_ap}
    Under the same assumptions as in Theorem \ref{the:finite-time-entropy}, there exists a constant $C$ independent of $t$ and $N$, such that for every $0 \leq t \leq T$,
    \begin{equation}\label{eq:drift_lln}
        \int_{\R^{3N}} F_N(t) \frac{1}{N} \sum_{i=1}^N \left| b \ast f_t(v^i)-\frac{1}{N} \sum_{j=1}^N b(v^i-v^j) \right|^2 \ud V \leq \frac{C(1+T)^5}{N} M_{16},
    \end{equation}
    \begin{equation}\label{eq:cancel_estimate}
        \int_{\R^{3N}} F_N(t) \frac{1}{N} \sum_{i=1}^N \left| a \ast f_t(v^i)-\frac{1}{N} \sum_{j=1}^N a(v^i-v^j) \right|^2 \ud V \leq \frac{C(1+T)^5}{N} M_{16},
    \end{equation}
    \begin{equation}\label{eq:weighted_lln_assumption}
        \int_{\R^{3N}} F_N(t) \frac{1}{N} \sum_{i=1}^N \left| \frac{1}{N} \sum_{j=1}^N a(v^i-v^j) \nabla \log f_t(v^j) -(a \ast (f_t \nabla \log f_t))(v^i) \right|^2 \ud V \leq \frac{C(1+T)^5}{N} M_{72}.
    \end{equation}
\end{prop}

\begin{rem}
    Actually \eqref{eq:drift_lln} and \eqref{eq:cancel_estimate} can be proved in a rather direct sense and bounded by slightly lower order moments by expanding the summation and repeatedly applying conservation laws. In \cite{carrillo2025relative} the law of large numbers is exactly \eqref{eq:cancel_estimate}, but the particle system dynamics is slightly different, and a proof with direct expansion is given there. Here we gather them together to show the generality of our duality argument. The essential estimate is indeed \eqref{eq:weighted_lln_assumption}, since the appearance of score function does not allow an explicit expansion formula.
\end{rem}

\begin{proof}
    We first prove \eqref{eq:drift_lln}. Recall the fluctuation function
    \begin{equation*}
        R_\psi^N(t)=\frac{1}{N} \sum_{j=1}^N \psi(V^j(t))- \int_{\R^3} \psi(v) f_t(v) \ud v.
    \end{equation*}
    In Step 4 of the proof of Lemma \ref{lem:score_weighted_moment_lln}, we have shown that $b \ast \mu_N(s)(v)-b \ast f_s(v)=2m_N$. Therefore, the quantity inside the square is equal to $R_\psi^N(t)$ with $\psi=-2z$. Applying Lemma \ref{lem:score_weighted_moment_lln} and noticing $n_0=1, n_1=n_2=0$, we immediately get \eqref{eq:drift_lln}.
    
    We next prove \eqref{eq:cancel_estimate}. In Step 4 of the proof of Lemma \ref{lem:score_weighted_moment_lln}, we have also shown that
    \begin{equation*}
        a \ast \mu_N(s) (v)- a \ast f_s(v)=-2(v \cdot m_N) \Id+v \otimes m_N+m_N \otimes v+(e_N-3) \Id-(Q(s)-E(s)).
    \end{equation*}
    Therefore, by the Cauchy--Schwarz inequality, we have the following entrywise estimate:
    \begin{align*}
        &\, \left| \left[ a \ast f_t(v^i)-\frac{1}{N} \sum_{j=1}^N a(v^i-v^j) \right]_{\alpha\beta} \right|^2 \\
        =&\, \left| \left( R_{|v|^2}^N(t)-2 v^i_\gamma R_{v_\gamma }^N(t) \right) \delta_{\alpha\beta}+ v^i_\alpha R_{v_\beta }^N(t)+ v^i_\beta R_{v_\alpha }^N(t)- R_{v_\alpha v_\beta}^N(t) \right|\\
        \leq &\, C(1+|v^i|^4) \sum_{\psi \in \mathcal S_t^1} \left| R_\psi^N(t) \right|^2,
    \end{align*}
    where the set of test functions $\mathcal S^1_t=\left\{v_\gamma, v_\gamma v_\eta: \gamma,\eta \in \{1,2,3\} \right\}$. Applying Lemma \ref{lem:score_weighted_moment_lln} again and noticing $n_0=2, n_1=1, n_2=0$, we get \eqref{eq:cancel_estimate}.
    
    Finally we prove \eqref{eq:weighted_lln_assumption}. For simplicity we denote by $p_t=\nabla \log f_t$ in this part. For each $\alpha\in\{1,2,3\}$, a careful expansion gives
    \begin{align*}
       \Bigg[ \frac{1}{N} \sum_{j=1}^N a(v^i-v^j) p_t(v^j)-(a \ast (f_tp_t)(v^i)) \Bigg]_\alpha &\, =\left( |v^i|^2 \delta_{\alpha\beta}- v^i_\alpha v^i_\beta \right) R_{p_\beta}^N(t)-2 v^i_\gamma R_{v_\gamma p_\alpha}^N(t)+ v^i_\alpha R_{v_\beta p_\beta}^N(t)\\
        &\, +v^i_\beta R_{v_\alpha p_\beta}^N(t)+ R_{|v|^2 p_\alpha}^N(t) -R_{v_\alpha v_\beta p_\beta}^N(t).  
    \end{align*}
    Therefore, by the Cauchy--Schwarz inequality, we have the following estimate:
    \begin{equation*}
        \left| \frac{1}{N} \sum_{j=1}^N a(v^i-v^j) p_t(v^j) -(a \ast (f_tp_t)(v^i)) \right|^2 \leq C(1+|v^i|^4) \sum_{\psi \in \mathcal S_t^2} \left| R_\psi^N(t) \right|^2,
    \end{equation*}
    where the set of test function $\mathcal S^2_t= \left\{ p_{\alpha}, v_\gamma p_{\alpha},v_\gamma v_\eta p_{\alpha}: \alpha, \gamma, \eta \in \{1,2,3\} \right\}$. And every $\psi \in \mathcal S^2_t$ satisfies \eqref{eq:psi_polynomial_bound} due to Theorem \ref{the:logarithmic_derivative}. Applying Lemma \ref{lem:score_weighted_moment_lln} again and this time noticing $n_0=3,n_1=4,n_2=5$ by Theorem \ref{the:logarithmic_derivative}, we get \eqref{eq:weighted_lln_assumption}.
\end{proof}

\section{Proof of Theorem \ref{the:finite-time-entropy}}\label{sec:finite_time}

In this section, we prove Theorem \ref{the:finite-time-entropy} by applying the law of large numbers in Section \ref{sec:lln} to the entropy decomposition in Lemma \ref{lem:evolution_entropy}. 

We recall from Lemma \ref{lem:evolution_entropy} the following time evolution formula of relative entropy.
\begin{align*}
    H_N \left( F_N(t)| f_t^{\otimes N} \right) &\, \leq H_N \left( F_N(0)| f_0^{\otimes N} \right) \\
    &\, +\int_0^t \int_{\R^{3N}} F_N(s) \frac{1}{N} \sum_{i=1}^N \left[ \frac{1}{N} \sum_{j=1}^N V_{f}(v^i,v^j) -\bar V_{f}(v^i) \right]\ud s,
\end{align*}
where the two-body mean-field defect functional $V_f$ is defined by
\begin{align*}
    V_{f}(v,w)= a(v-w): \frac{\nabla^2 f}{f}(v)&\, +2b(v-w) \cdot \nabla \log f(v)\\
    &\, -a(v-w):\nabla \log f(v) \otimes \nabla \log f(w),
\end{align*}
and $\bar V_f$ is the one-side expectation of $V_f$ under $f$:
\begin{equation*}
  \bar  V_{f}(v)=\int_{\R^3} V_{f}(v,w)f(w) \ud w.
\end{equation*}
{\bf Step 1: Decomposition of the defect functional.} We split $V_{f}=V_{f}^1+V_{f}^2+V_{f}^3$, where
\begin{equation*}
    V_{f}^1(v,w)=a(v-w):\frac{\nabla^2 f}{f}(v), \qquad V_{f}^2(v,w)=2b(v-w) \cdot \nabla \log f(v),
\end{equation*}
\begin{equation*}
    V_{f}^3(v,w)=-a(v-w): \nabla \log f(v) \otimes \nabla \log f(w).
\end{equation*}
For $\ell=1,2,3$, we define the corresponding marginal for $V_f^\ell$:
\begin{equation*}
    \bar V_{f}^\ell(v)=\int_{\R^3} V_{f}^\ell(v,w) f(w) \ud w
\end{equation*}
and the corresponding integral:
\begin{equation*}
    \mathcal I_\ell(s)=\frac{1}{N} \sum_{i=1}^N \int_{\R^{3N}} F_N(s) \left[ \frac{1}{N} \sum_{j=1}^N \left( V_{f}^\ell(v^i,v^j)-\bar V_{f}^\ell(v^i) \right) \right] \ud V .
\end{equation*}
Therefore, Lemma \ref{lem:evolution_entropy} implies
\begin{equation}\label{eq:entropy_by_Iell}
    H_N(F_N(t)|f_t^{\otimes N}) \leq H_N(F_N(0)|f_0^{\otimes N})+ \int_0^t \left( \mathcal I_1(s)+\mathcal I_2(s)+\mathcal I_3(s) \right) \ud s.
\end{equation}
{\bf Step 2: Estimate of $\mathcal I_1$.}
By the elementary identity
\begin{equation*}
    \frac{\nabla^2 f_s}{f_s}=\nabla^2 \log f_s+\nabla \log f_s \otimes \nabla \log f_s,
\end{equation*}
and thus the logarithmic derivative estimates in Theorem \ref{the:logarithmic_derivative} imply
\begin{equation*}
    \left| \frac{\nabla^2 f_s}{f_s}(v) \right| \leq C(1+s)(1+|v|^2).
\end{equation*}
Using the Cauchy--Schwarz inequality, and applying \eqref{eq:cancel_estimate} and Lemma \ref{lem:particle_moment_propagation} with $k=2$, we obtain
\begin{align}\label{eq:I1_bound}
    |\mathcal I_1(s)| &\, \leq \left[ \int F_N(s) \frac{1}{N} \sum_{i=1}^N \left| \frac{1}{N} \sum_{j=1}^N a(v^i-v^j)-a \ast f_s(v^i) \right|^2 \ud V \right]^{1/2} \cdot \left[ \int F_N(s) \frac{1}{N} \sum_{i=1}^N \left| \frac{\nabla^2 f_s}{f_s}(v^i) \right|^2 \ud V \right]^{1/2} \nonumber \\
    &\, \leq \frac{C(1+T)^4}{\sqrt N} (M_{16}+1).
\end{align}
{\bf Step 3: Estimate of $\mathcal I_2$.} We have the explicit formula
\begin{equation*}
    \mathcal I_2(s)=2 \int_{\mathbb R^{3N}} F_N(s) \frac{1}{N} \sum_{i=1}^N \nabla \log f_s(v^i) \cdot \left[ \frac{1}{N} \sum_{j=1}^N b(v^i-v^j)-(b \ast f_s)(v^i) \right] \ud V.
\end{equation*}
Using the Cauchy--Schwarz inequality, and applying \eqref{eq:drift_lln} and the logarithmic
growth bound on $\nabla \log f$ in Theorem \ref{the:logarithmic_derivative}, we obtain
\begin{align}\label{eq:I2_bound}
    |\mathcal I_2(s)| &\, \leq 2 \left( \int F_N(s) \frac{1}{N} \sum_{i=1}^N |\nabla \log f_s(v^i)|^2 \ud V \right)^{1/2} \cdot \left( \int F_N(s) \frac{1}{N} \sum_{i=1}^N \left| \frac{1}{N} \sum_{j=1}^N b(v^i-v^j)-(b \ast f_s)(v^i) \right|^2 \ud V \right)^{1/2} \nonumber \\
    &\, \leq \frac{C(1+T)^4}{\sqrt N} (M_{16}+1).
\end{align}
{\bf Step 4: Estimate of $\mathcal I_3$.} Since $(a \ast (f_t \nabla \log f_t))(v)=(b \ast f_t)(v)=-2v,$ we have
\begin{equation*}
    \bar V_f^3(v)=-\nabla \log f_t(v) \cdot ( a \ast ( f_t \nabla \log f_t))(v)=2v \cdot \nabla \log f_t(v).
\end{equation*}
Consequently, the fluctuation part is given by 
\begin{equation*}
    \frac{1}{N} \sum_{j=1}^N \left( V_f^3(v^i,v^j)-\bar V_f^3(v^i)\right)=- \nabla \log f_t(v^i) \cdot \left( \frac{1}{N} \sum_{j=1}^N a(v^i-v^j) \nabla \log f_t(v^j)-(a \ast (f_t \nabla \log f_t))(v^i) \right).
\end{equation*}
Using the Cauchy--Schwarz inequality, and applying \eqref{eq:weighted_lln_assumption} in Proposition \ref{prop:weighted_lln_ap} and the logarithmic growth bound on $\nabla \log f$, we obtain
\begin{align}\label{ineq: I3}
    \left| \mathcal I_3(t) \right| &\, \leq \left( \int_{\R^{3N}} F_N(t) \frac{1}{N} \sum_{i=1}^N |\nabla \log f_t (v^i)|^2 \ud V \right)^{1/2} \nonumber \\
    &\, \times \left( \int_{\R^{3N}} F_N(t) \frac{1}{N} \sum_{i=1}^N \left| \frac{1}{N} \sum_{j=1}^N a(v^i-v^j) \cdot \nabla \log f_t(v^j) -(a \ast (f_t \nabla \log f_t))(v^i) \right|^2 \ud V \right)^{1/2} \nonumber \\
    &\, \leq \frac{C(1+T)^4}{\sqrt{N}} (M_{72}+1).
\end{align}
{\bf Step 5: Conclusion.} Since we take $F_N(0)=f_0^{\otimes N}$, the initial relative entropy is zero. Combining \eqref{eq:entropy_by_Iell} with \eqref{eq:I1_bound} \eqref{eq:I2_bound} \eqref{ineq: I3} yields, for every $0 \leq t \leq T$,
\begin{equation*}
    H_N(F_N(t)| f_t^{\otimes N}) \leq \int_0^t \left( |\mathcal I_1(s)|+|\mathcal I_2(s)|+|\mathcal I_3(s)| \right) \ud s \leq \frac{C(1+T)^5}{\sqrt N} (M_{72}+1).
\end{equation*}
This proves Theorem \ref{the:finite-time-entropy}.

\section{Uniform-in-time propagation of chaos}\label{sec:uniform_time}

In this section, we prove Theorem \ref{the:time_uniform} and Corollary \ref{cor:uniform-in-time-poc}. We first decompose the relative entropy as
\begin{equation*}
    H_N(F_N|f^{\otimes N})=\frac{1}{N} \int_{\R^{3N}} F_N \log \frac{F_N}{\gamma_N} \ud V-\frac{1}{N} \int_{\R^{3N}} F_N \log \frac{f^{\otimes N}}{\gamma^{\otimes N}} \ud V,
\end{equation*}
where $\gamma$ represents the standard Gaussian distribution
\begin{equation*}
    \gamma(v)=(2\pi)^{-3/2}e^{-|v|^2/2}, \quad \gamma_N=\gamma^{\otimes N}.
\end{equation*}
We first deal with the first term, which is exactly $H_N(F_N|\gamma_N)$.

In this section, we denote by
\begin{equation*}
    I(\rho)=\int |\nabla\rho|^2/\rho
\end{equation*}
the unnormalized absolute Fisher information, in the ambient dimension of $\rho$. 

\begin{lem}\label{lem:FN_gammaN_uniform}
    Assume the hypotheses of Theorem \ref{the:time_uniform}. For any $\delta>0$, there exists a constant $C_\delta>0$ independent of $t$ and $N$, such that for every $t \geq 0$,
    \begin{equation*}
        H_N(F_N|\gamma_N)=\frac{1}{N} \int_{\R^{3N}} F_N \log \frac{F_N}{\gamma_N} \ud V \leq C_\delta \left( \frac{1}{\sqrt N}+\frac{1}{(1+t)^\delta} \right).
    \end{equation*}
\end{lem}

\begin{proof}
    \emph{Step 1. Bounding the entropy by the HWI inequality.}
    We apply the HWI inequality for Gaussian variable, see \cite[Corollary 9.3.3]{bakry2013analysis}, and obtain
    \begin{align*}
        H_N(F_N|\gamma_N) \leq \frac{\cW_2(F_N,\gamma_N)}{\sqrt N} \left( I_N(F_N|\gamma_N)\right)^{1/2}.
    \end{align*}
    Here \(\cW_2\) denotes the quadratic Wasserstein distance on \(\R^{3N}\), and
    \begin{equation*}
        I_N(F_N|\gamma_N)= \frac{1}{N} \sum_{i=1}^N \int_{\R^{3N}} F_N \left| \nabla_{v^i} \log \frac{F_N}{\gamma_N} \right|^2 \ud V
    \end{equation*}
    is the normalized relative Fisher information between $F_N$ and $\gamma^{\otimes N}$. Since
    \begin{equation*}
        \nabla_{v^i} \log \left( \frac{F_N}{\gamma_N} \right)=\nabla_{v^i} \log F_N+v^i,
    \end{equation*}
    we deduce the following identity
    \begin{equation*}
        I_N(F_N|\gamma_N)=\frac{1}{N} I(F_N)-6+\frac{1}{N} \sum_{i=1}^N \int_{\R^{3N}} |v^i|^2 F_N \ud V.
    \end{equation*}
    By monotonicity of the Fisher information for the particle system (see \cite{carrillo2025fisher}) and conservation of kinetic energy, we have
    \begin{equation*}
        I_N(F_N|\gamma_N) =\frac{1}{N} I(F_N)-3 \leq I(f_0)-3 \leq I(f_0).
    \end{equation*}
    Therefore, it only remains to estimate the normalized Wasserstein distance
    \begin{equation*}
        \frac{\cW_2(F_N,\gamma_N)}{\sqrt N}.
    \end{equation*}

    \noindent\emph{Step 2. Control of \(\cW_2(F_N,\gamma_N)\).}
    We introduce the Boltzmann sphere
    \begin{equation*}
        \mathcal{K}^N= \left\{ V \in \R^{3N}: \sum_{i=1}^N |v_i|^2 = 3N, \quad \sum_{i=1}^N v_i = 0 \right\}.
    \end{equation*}
    We now renormalize the particle system onto \(\mathcal{K}^N\). Hereinafter we omit the time index $t$ for simplicity. Define the renormalized velocity vector
    \begin{equation*}
        \widehat{\bV}=(\widehat V^1,\cdots,\widehat V^N), \quad \widehat V^i=\frac{V^i-m(\bV)}{S(\bV)},
    \end{equation*}
    where we denote by
    \begin{equation*}
        m(\bV)=\frac{1}{N} \sum_{i=1}^N V^i, \quad S^2(\bV)=\frac{1}{3N} \sum_{i=1}^N |V^i-m(\bV)|^2
    \end{equation*}
    the pointwise conserved average momentum and variance. The renormalized system $\widehat{\bV}$ is supported on the Boltzmann sphere $\mathcal{K}^N$. By pointwise conservation of momentum and kinetic energy, \(m(\bV)\) and \(S(\bV)\) are constant and thus equal to their initial values \(m(\bV_0)\) and \(S(\bV_0)\). We simply write them as \(m\) and \(S\).
    
    Let \(\widehat F_N\) denote the law of \(\widehat{\bV}\), and define \(\widehat\gamma_N\) analogously from \(\gamma^{\otimes N}\). Then by the triangle inequality,
    \begin{equation}\label{ineq:tri_decom}
        \frac{\cW_2(F_N,\gamma_N)}{\sqrt N} \leq \frac{\cW_2(F_N,\widehat F_N)}{\sqrt N}+ \frac{\cW_2(\widehat F_N,\widehat\gamma_N)}{\sqrt N}+\frac{\cW_2(\widehat\gamma_N,\gamma_N)}{\sqrt N}.
    \end{equation}
    The first term on the right-hand side can be estimated directly. The normalization map $\bV \mapsto \widehat{\bV}$ provides a transport plan. Since $S \neq 0$ almost surely, we have
    \begin{align*}
	    \frac{\cW_2^2(F_N,\widehat F_N)}{N} \leq \frac{1}{N} \E \left[ \sum_{i=1}^N |V_t^i-\widehat V_t^i|^2 \right] =\frac{1}{N} \E \left[ \sum_{i=1}^N \left|(1-S^{-1})V_t^i+mS^{-1} \right|^2 \right]=3\E (S-1)^2+\E |m|^2.
    \end{align*}
    Since the random variables \(V_0^i\) are i.i.d. with mean zero and second moment \(3\), and have finite fourth moments, we obtain
    \begin{align*}
	    \E (S-1)^2 \leq \E (S^2-1)^2 \leq \frac{2}{9N^2} \E \left[ \left| \sum_{i=1}^N (|V^i(0)|^2-3) \right|^2 \right]+ \frac{2}{9} \E |m|^4 \leq \frac{C}{N} (M_4+1).
    \end{align*}
    Moreover, we have \(\E[|m|^2]=3/N\) by direct computations. Hence, we have
    \begin{equation*}
        \frac{\cW_2^2(F_N,\widehat F_N)}{N} \leq \frac{C}{N}.
    \end{equation*}
    By the same argument, we have
    \begin{equation*}
        \frac{\cW_2^2(\widehat\gamma_N,\gamma_N)}{N} \leq \frac{C}{N}.
    \end{equation*}

    On the other hand, \(\widehat\gamma_N\) is exactly the uniform probability measure on the Boltzmann sphere \(\mathcal{K}^N\) by \cite[Proof of Lemma 4.3]{rousset2014n}. Therefore, by \cite[Theorem 24]{fournier2017kac}, we obtain
    \begin{equation*}
        \frac{\cW_2^2(\widehat F_N,\widehat\gamma_N)}{N} \leq \frac{C_\delta}{(1+t)^{2\delta}} \left( 1+\int_{\R^3} |v|^{8+8\delta} f_0(v) \ud v \right).
    \end{equation*}
    Combining this estimate with \eqref{ineq:tri_decom}, we conclude that
    \begin{equation*}
        \frac{\cW_2(F_N,\gamma_N)}{\sqrt N} \leq C_\delta \left( \frac{1}{\sqrt N}+\frac{1}{(1+t)^\delta} \right),
    \end{equation*}
    which completes the proof.
\end{proof}

We now turn to the second term, which writes
\begin{equation*}
    \frac{1}{N} \int_{\R^{3N}} F_N \log \frac{f^{\otimes N}}{\gamma_N} \ud V.
\end{equation*}
We first establish a propagation-of-exponential-moments estimate for the particle system \eqref{eq:particle-sde}.

\begin{prop}[Propagation of exponential moments]\label{prop:landau_exponential_moments}
    Assume that $f_0$ is exponentially integrable:
    \begin{equation*}
        \int_{\R^3} e^{\sigma_0 |v|^2} f_0(v) \ud v<\infty
    \end{equation*}
    for some \(\sigma_0>0\). Then there exists \(a_0>0\) such that
    \begin{equation*}
        \sup_{t \geq 0} \E \left[ e^{a_0|V^1(t)|^2} \right]= \sup_{t \geq 0}\int_{\R^3} e^{a_0|v|^2} F_N^1(t,v) \ud v \leq 4.
    \end{equation*}
    More precisely, one may take
    \begin{equation*}
        a_0=\frac{\sigma_0}{2e \left( \E [e^{\sigma_0|V^1(0)|^2}]+1 \right)}.
    \end{equation*}
\end{prop}

\begin{proof}
    For simplicity, we denote by
    \begin{equation*}
        K_0 =\E \left[ e^{\sigma_0|V^1(0)|^2} \right]<\infty,
    \end{equation*}
    and recall that $e_N$ represents the average kinetic energy. For any \(k\ge 1\), we define
    \begin{equation*}
        m_k(t)=\E \left[ |V^1(t)|^{2k} | \bV_0 \right].
    \end{equation*}
    By \eqref{eq:conditional_moment_bound} in Lemma \ref{lem:particle_moment_propagation}, we already have 
    \begin{equation*}
        m_k(t) \leq \max \left\{ |V^1(0)|^{2k}, (ke_N)^k \left( 1+\frac{k-1}{N} \right)^{-k}\right\}.
    \end{equation*}
    We now pass to the exponential moment by Taylor's expansion. For any \(c\in(0,\sigma_0]\), we compute
    \begin{align*}
        \E \left[ e^{c|V^1(t)|^2} |\bV_0 \right] &\, = \sum_{k=0}^\infty \frac{c^k}{k!} m_k(t)\\
        &\, \leq e^{c|V^1(0)|^2} +\sum_{k=1}^\infty \frac{(cke_N)^k}{k!} \left( 1+\frac{k-1}{N} \right)^{-k}.
    \end{align*}
    We take expectation on both sides. The first term on the right-hand side is bounded by the initial assumption. It remains to control the series
    \begin{equation*}
        \sum_{k=1}^\infty \E \left[ \frac{(cke_N)^k}{k!} \left( 1+\frac{k-1}{N} \right)^{-k} \right].
    \end{equation*}
    To this end, note that for \(0 \leq s \leq 1\) and \(y \geq 0\), we have
    \begin{equation*}
        e^{sy} \leq 1-s+se^y
    \end{equation*}
    by convexity of the exponential function. Hence, for any \(0 \leq \ell \leq N\),
    \begin{align*}
        \E \left[ e^{\ell \sigma_0 e_N} \right] &\, = \left( \E \left[ e^{\frac{\ell}{N}\sigma_0 |V^1(0)|^2} \right] \right)^N \\
        &\, \leq \left( 1+\frac{\ell}{N} (K_0-1) \right)^N \leq \exp \left( \ell(K_0-1) \right).
    \end{align*}
    Using the elementary bound \(y^k \leq k! e^{\ell y} \ell^{-k}\), we deduce that
    \begin{equation*}
        \E \left[ (\sigma_0 e_N)^k \right] \leq \exp \left( \ell (K_0-1) \right) k! \ell^{-k},\quad 0 \leq \ell \leq N.
    \end{equation*}
    Therefore, for each $k \geq 1$, we bound the expectation by
    \begin{equation}\label{eq:series_bound} 
        \E \left[ \frac{(cke_N)^k}{k!} \left( 1+\frac{k-1}{N} \right)^{-k} \right] \leq \exp \left( \ell(K_0-1) \right) \left( \frac{ck}{\ell \sigma_0} \right)^k \left(1+\frac{k-1}{N} \right)^{-k}.
    \end{equation}
    Now we need to determine the choice of $\ell$ according to the value of $k$. We split into two cases.

    \noindent\textbf{Case 1: \(k \leq N(K_0-1)\).}
    Choose \(\ell=\frac{k}{K_0-1} \leq N\). From \eqref{eq:series_bound}, we have
    \begin{equation*}
        \E \left[ \frac{(cke_N)^k}{k!} \left( 1+\frac{k-1}{N} \right)^{-k} \right] \leq \left( \frac{ce(K_0-1)}{\sigma_0} \right)^k.
    \end{equation*}

    \noindent\textbf{Case 2: \(k> N(K_0-1)\).}
    Choose \(\ell=N\). Again by \eqref{eq:series_bound}, we have
    \begin{align*}  
        \E \left[ \frac{(cke_N)^k}{k!} \left( 1+\frac{k-1}{N} \right)^{-k} \right] &\, \leq e^{N(K_0-1)} \left( \frac{ck}{N\sigma_0} \right)^k \left( 1+\frac{k-1}{N} \right)^{-k} \\
        &\, \leq e^{N(K_0-1)} \left( \frac{c}{\sigma_0} \right)^k.
    \end{align*}
    Since \(k>N(K_0-1)\), we have \(e^{N(K_0-1)}<e^k\), so the final bound is
    \begin{equation*}
        \E \left[ \frac{(cke_N)^k}{k!} \left( 1+\frac{k-1}{N} \right)^{-k} \right] \leq \left( \frac{ce}{\sigma_0} \right)^k.
    \end{equation*}

    With the two bounds above, it suffices to choose $c=a_0$ appropriately to make the series converge. An explicit choice is given by
    \begin{equation*}
        a_0=\frac{\sigma_0}{2e(K_0+1)}.
    \end{equation*}
    Then obviously we have
    \begin{equation*}
        \frac{a_0e(K_0-1)}{\sigma_0}=\frac{K_0-1}{2(K_0+1)} \leq \frac{1}{2}, \quad \frac{a_0e}{\sigma_0}= \frac{1}{2(K_0+1)} \leq \frac{1}{2}.
    \end{equation*}
    Hence the contribution of each case is bounded by a geometric series:
    \begin{equation*}
        \sum_{1 \leq k \leq N(K_0-1)} \E \left[ \frac{(a_0ke_N)^k}{k!} \left( 1+\frac{k-1}{N} \right)^{-k} \right] \leq \sum_{k \geq 1} 2^{-k} \leq 1,
    \end{equation*}
    and similarly
    \begin{equation*}
        \sum_{k> N(K_0-1)} \E \left[ \frac{(a_0ke_N)^k}{k!} \left( 1+\frac{k-1}{N} \right)^{-k} \right] \leq \sum_{k \geq 1} 2^{-k} \leq 1.
    \end{equation*}
    Therefore, we conclude by
    \begin{equation*}
        \sum_{k=1}^\infty \E \left[ \frac{(a_0ke_N)^k}{k!} \left( 1+\frac{k-1}{N} \right)^{-k} \right] \leq 2.
    \end{equation*}
    It remains to estimate the initial term. Since \(a_0 \leq \sigma_0\), the same convexity inequality gives
    \begin{equation*}
        \E \left[ e^{a_0|V^1(0)|^2} \right]= \E \left[ e^{\frac{a_0}{\sigma_0} \sigma_0|V^1(0)|^2} \right] \leq 1+\frac{a_0}{\sigma_0}(K_0-1) \leq 2.
    \end{equation*}
    Combining the previous bounds together, we finally obtain
    \begin{equation*}
        \sup_{t \geq 0} \E \left[ e^{a_0|V^1(t)|^2} \right] \leq 2+2=4.
    \end{equation*}
    This completes the proof.
\end{proof}

We now complete our estimate of the second term, which is
\begin{equation*}
    -\frac{1}{N} \int_{\R^{3N}} F_N \log \frac{f^{\otimes N}}{\gamma_N} \ud V.
\end{equation*}
By symmetry, it is equal to
\begin{equation*}
    -\int_{\R^3} F_N^1 \log \frac{f}{\gamma} \ud v.
\end{equation*}
We will prove the following estimate.

\begin{lem}\label{lem:cross_term_decay}
    Under the assumptions of Theorem \ref{the:time_uniform}, there exists a constant \(C_0>0\), depending only on the initial bounds in \eqref{eq:initial-gaussian-bounds} and \eqref{eq:relative_L2_initial}, such that for all \(t \geq 0\),
    \begin{equation*}
        -\int_{\R^3} F_N^1 \log \frac{f}{\gamma} \ud v \leq C_0 e^{-C_0^{-1}t}.
    \end{equation*}
\end{lem}

\begin{proof}
    We split the integral into two parts according to the value of $v$:
    \begin{equation*}
        -\int_{\R^3} F_N^1 \log \frac{f}{\gamma} \ud v= -\int_{|v| \geq R} F_N^1 \log \frac{f}{\gamma} \ud v -\int_{|v|<R} F_N^1 \log \frac{f}{\gamma} \ud v,
    \end{equation*}
    where \(R \ge 1\) will be determined later. From \eqref{eq:initial-gaussian-bounds} and standard well-posedness results (see for instance \cite{villani1998spatially}), we have $-C-C|v|^2 \leq \log f \leq C$ for any $t$, which implies,
    \begin{equation*}
        \left| \log \frac{f}{\gamma}(v) \right| \leq C(1+|v|^2).
    \end{equation*}
    By the initial assumption \eqref{con:gaussian-upper-bound}, we can apply Proposition \ref{prop:landau_exponential_moments} to obtain the tail estimate
    \begin{align}\label{ineq:FfR_rewrite}
        \left| \int_{|v| \geq R} F_N^1 \log \frac{f}{\gamma} \ud v \right| &\, \leq C \int_{|v|\geq R} |v|^2 F_N^1 \ud v \nonumber\\
        &\, \leq \frac{2C}{a_0} e^{-a_0R^2/2} \int_{|v| \geq R} e^{a_0|v|^2} F_N^1 \ud v \leq \frac{8C}{a_0} e^{-a_0R^2/2}.
    \end{align}

    We next estimate the local term. Since \(f \geq c_0 e^{-c_0^{-1}|v|^2}\) for some \(c_0\in(0,1)\), for any \(R \geq 1\) and \(|v| \leq R\), we have
    \begin{equation*}
        \left| \log f(v)-\log \gamma(v) \right| \leq e^{c_0^{-1}R^2} |f(v)-\gamma(v)|.
    \end{equation*}
    Consequently, by the Cauchy--Schwarz inequality, we have
    \begin{align}
        -\int_{|v|<R} F_N^1 \log \frac{f}{\gamma} \ud v &\, \leq e^{c_0^{-1}R^2} \int_{|v|<R} F_N^1 \left| \frac{f-\gamma}{\gamma} \right| \gamma \ud v \nonumber\\
        &\, \leq e^{c_0^{-1}R^2} \left( \int_{\R^3} (F_N^1)^2 \gamma \ud v \right)^{1/2} \left( \int_{\R^3} \left| \frac{f-\gamma}{\gamma} \right|^2 \gamma \ud v \right)^{1/2}.
    \label{ineq:Ffgamma_rewrite}
    \end{align}
    It remains to bound the two factors on the right-hand side. Applying the fact that $\gamma \leq 1$, the Sobolev embedding applied to $\sqrt{F_N^1}$ and the monotonicity of the Fisher information, we get
    \begin{equation*}
        \int_{\R^3} (F_N^1)^2 \gamma \ud v \leq \|F_N^1\|_{L^2}^2 \leq \|F_N^1\|_{L^3}^{3/2} \leq CI(F_N^1)^{3/2} \leq CI(f_0)^{3/2}.
    \end{equation*}
    For the second factor, we recall that \eqref{eq:relative_L2_initial} gives such a bound for the initial data, and by \cite[Theorem 1]{caja2025contractivity} we get
    \begin{equation*}
        \int_{\R^3} \left| \frac{f-\gamma}{\gamma} \right|^2 \gamma \ud v \leq C_0 e^{-4t},
    \end{equation*}
    where $C_0$ depends only on the initial bound in \eqref{eq:relative_L2_initial}. Substituting these two bounds into \eqref{ineq:Ffgamma_rewrite}, and absorbing polynomial factors into the exponential, we obtain
    \begin{equation*}
        -\int_{|v|<R} F_N^1 \log \frac{f}{\gamma} \ud v \leq C_0 e^{c_0^{-1}R^2} e^{-2t}.
    \end{equation*}
    Together with \eqref{ineq:FfR_rewrite}, this gives, for any \(R \geq 1\),
    \begin{equation*}
        -\int_{\R^3} F_N^1 \log \frac{f}{\gamma} \ud v \leq C_0 e^{c_0^{-1}R^2} e^{-2t}+ \frac{8C}{a_0} e^{-a_0 R^2/2}.
    \end{equation*}
    We now optimize in $R$ by taking \(R^2=c_0 t+1\). Then we have
    \begin{equation*}
        -\int_{\R^3} F_N^1 \log \frac{f}{\gamma} \ud v \leq C_0 e^{-t}+ \frac{8C}{a_0} e^{-a_0c_0 t/2}.
    \end{equation*}
    After enlarging the constant \(C_0\), this completes the proof.
\end{proof}

Combining Lemma \ref{lem:FN_gammaN_uniform} and Lemma \ref{lem:cross_term_decay}, and bounding the exponential term by a negative power, we have proved Theorem \ref{the:time_uniform}.

\begin{proof}[Proof of Corollary \ref{cor:uniform-in-time-poc}]
    Now we prove Corollary \ref{cor:uniform-in-time-poc}.
    
    {\bf Step 1: Interpolation.} Applying Theorem \ref{the:finite-time-entropy} with  $T= N^{\varepsilon/4}-1$, we obtain
    \begin{equation*}
        \sup_{0 \leq t \leq N^{\varepsilon/4}-1} H_N \left( F_N(t)| f_t^{\otimes N} \right) \leq C_\varepsilon N^{-1/2+\varepsilon}.
    \end{equation*}
    On the other hand, applying Theorem \ref{the:time_uniform} with $\delta=2/\varepsilon>0$, we get 
    \begin{equation*}
        \sup_{t \geq N^{\varepsilon/4}-1} H_N \left( F_N(t)| f_t^{\otimes N} \right) \leq C_\varepsilon N^{-1/2}.
	\end{equation*}
    Combining the preceding two estimates and enlarging the constant if necessary, we immediately get \eqref{ineq: entropy uni}.

    {\bf Step 2: Proof of Wasserstein-1 Talagrand inequality.} For any $x,y \in \R^{3N}$, we define the additive product distance
    \begin{equation*}
        d_N(x,y)=\sum_{i=1}^N|x_i-y_i|_{\mathbb R^3},
    \end{equation*}
    and let $\mathbf W_{1}$ denote the Wasserstein-$1$ distance associated with $d_N$. In this step, we are going to prove a Talagrand $T_1$ type inequality, namely 
    \begin{equation*}
        N^{-2 }\mathbf{W}^2_1 \left( F_N(t),f_t^{\otimes N} \right) \leq C H_N \left( F_N(t)| f_t^{\otimes N} \right) \leq C_\varepsilon N^{-1/2+\varepsilon}.
    \end{equation*}
    We first establish a one-particle exponential estimate. Let $h: \R^3 \to \R$ be any $1$-Lipschitz test function. By the initial condition \eqref{con:gaussian-upper-bound} and \cite[Section 6]{villani1998spatially}, there exists $\sigma_0>0$ depending only on $L_0$ such that
    \begin{equation*}
        \int_{\R^3} e^{\sigma_0|v|^2} f_t(v) \ud v \leq C
    \end{equation*}
    uniformly in $t$. We claim that there exists $C_0>0$, depending only on $\sigma_0$, such that
    \begin{equation*}
        \log \int_{\mathbb R^3} e^{\lambda(h- \int h f_t)}\,f_t(v) \ud v \leq C_0 \lambda^2
    \end{equation*}
    for every $\lambda \in \R$. Indeed, we sample $X,Y$ as independent random variables with law $f_t$.

    {\bf Case 1:} If $\lambda^2 \leq \sigma_0$, by Jensen's inequality, we have
    \begin{equation*}
        \int_{\R^3} e^{\lambda (h-\int h f_t)} f_t(v) \ud v = \E \left[ e^{ \E [\lambda (h(X)-h(Y))|X]} \right] \leq \E \left[ e^{\lambda (h(X)-h(Y))} \right].
    \end{equation*}
    By symmetry, we can also exchange $X$ and $Y$ and apply $\cosh x \leq e^{x^2/2}$ to get
    \begin{equation*}
        \leq \E \left[ \cosh \left( \lambda (h(X)-h(Y)) \right) \right] \leq \E \left[ e^{\frac{\lambda^2}{2}(h(X)-h(Y))^2} \right].
    \end{equation*}
    Since $h$ is $1$-Lipschitz, we further bound it by
    \begin{equation*}
        \E \left[ e^{\frac{\lambda^2}{2}|X-Y|^2} \right] \leq \E \left[ e^{\lambda^2 (|X|^2+|Y|^2)} \right] = \left( \E \left[ e^{\lambda^2 |X|^2} \right] \right)^2 \leq \E \left[ e^{\sigma_0|X|^2} \right]^{2\lambda^2/\sigma_0}.
    \end{equation*}
    This gives the claim.

    {\bf Case 2:} If $\lambda^2\ge\sigma_0$, since $h$ is $1$-Lipschitz, we write
    \begin{equation*}
        \left| h(v)-\int_{\R^3} h(w) f_t(w) \ud w \right| \leq |v|+\int_{\R^3} |w|f_t(w) \ud w.
    \end{equation*}
    Using the elementary inequality $|\lambda||v| \leq \sigma_0 |v|^2+\frac{\lambda^2}{4\sigma_0}$, we obtain
    \begin{equation*}
        \log \int_{\R^3} e^{\lambda (h-\int h f_t)} f_t(v) \ud v \leq |\lambda| \int_{\R^3} |v| f_t(v) \ud v +\frac{\lambda^2}{4\sigma_0} +\log \int_{\R^3} e^{\sigma_0|v|^2} f_t(v) \ud v \leq C_0 \lambda^2,
    \end{equation*}
    which gives the claim.

    We now adapt a tensorized version of the previous claim. Let $X_1,\cdots,X_N$ be independent random variables with common law $f_t$, and let $\Phi:\R^{3N} \to \R$ be $1$-Lipschitz with respect to $d_N$. For each $1 \leq i \leq N$, we define
    \begin{align*}
        \mathcal F_i= \sigma(X_1,\cdots,X_i), \quad S_i= \E \left[ \Phi(X_1,\cdots,X_N)|\mathcal F_i \right], \quad D_i= S_i-S_{i-1},
    \end{align*}
    For fixed $X_1,\ldots,X_{i-1}$, we also define
    \begin{equation*}
        h_i(v)=\int_{(\R^3)^{N-i}} \Phi (X_1, \cdots, X_{i-1}, v, v_{i+1}, \cdots, v_N) f_t^{\otimes(N-i)} \ud v_{i+1} \cdots \ud v_N.
    \end{equation*}
    Since $\Phi$ is $1$-Lipschitz with respect to $d_N$, we know $h_i$ is also $1$-Lipschitz. Therefore, by the previous claim, we have
    \begin{equation*}
        \E \left[ e^{\lambda D_i}|\mathcal F_{i-1} \right] =\int_{\R^3} e^{\lambda( h_i(v)-\int_{\R^3} h_i f_t \ud v)} f_t \ud v \leq e^{C_0\lambda^2}.
    \end{equation*}
    Iterating this conditional estimate yields
    \begin{equation*}
        \log \int_{\R^{3N}} e^{\lambda \left( \Phi-\int \Phi f_t^{\otimes N} \right)} f_t^{\otimes N} \ud V \leq C_0 N\lambda^2.
    \end{equation*}

    Now we prove our desired Talagrand inequality. By Kantorovich duality, it suffices to estimate
    \begin{equation*}
        \int_{\R^{3N}} \Phi (F_N(t)-f_t^{\otimes N}) \ud V
    \end{equation*}
    for any $1$-Lipschitz test function $\Phi$. Using the Donsker--Varadhan variational inequality, we obtain
    \begin{equation*}
        \int_{\R^{3N}} \Phi (F_N(t)-f_t^{\otimes N}) \ud V \leq \frac{N}{\eta} H_N \left( F_N(t)| f_t^{\otimes N} \right)+ \frac{1}{\eta} \log \int_{\R^{3N}} e^{\eta \left( \Phi-\int \Phi f_t^{\otimes N} \right)} f_t^{\otimes N} \ud V.
    \end{equation*}
    We take $\eta=\sqrt{ H_N \left( F_N(t)| f_t^{\otimes N} \right)}$. Combining with the previous estimate, we get
    \begin{equation*}
        \int_{\R^{3N}} \Phi (F_N(t)-f_t^{\otimes N}) \ud V \leq C_0 N \sqrt{H_N \left( F_N(t)| f_t^{\otimes N} \right)}.
    \end{equation*}
    Since this inequality holds for any $1$-Lipschitz test function $\Phi$ and $C_0$ is independent of $\Phi$, we conclude our desired Talagrand type inequality.

    {\bf Step 3: Conclusion.} We now proceed to prove \eqref{ineq: was2}. For each $t \geq 0$, let
    \begin{equation*}
        \mathbf{Z}(t)=\left( Z^1(t), \cdots, Z^N(t) \right), \quad \mathbf{W}(t)=\left( W^1(t), \cdots, W^N(t) \right)
    \end{equation*}
	be an optimal coupling of $F_N(t)$ and $f_t^{\otimes N}$ in $\mathbf{W}_1$. Thus by definition, $\mathbf{Z}(t) \overset{\text{law}}{=} \mathbf{V}(t)$ and
    \begin{equation*}
        \E \left| \mathbf{Z}(t)-\mathbf{W}(t) \right| =\mathbf{W}_{1} \left( F_N(t),f_t^{\otimes N} \right).
    \end{equation*}
    Using the triangle inequality, we get
    \begin{align*}
        \E \left[ \mathcal{W}_1 \left( N^{-1} \sum_{i=1}^N \delta_{V^i(t)},f_t \right) \right]= &\, \E \left[ \mathcal{W}_1 \left( N^{-1} \sum_{i=1}^N \delta_{Z^i(t)},f_t \right) \right] \\
        \leq &\,  \E \left[ \mathcal{W}_1 \left( N^{-1} \sum_{i=1}^N \delta_{Z^i(t)},N^{-1} \sum_{i=1}^N \delta_{W^i(t)} \right) \right]+\E \left[ \mathcal{W}_1 \left( N^{-1} \sum_{i=1}^N \delta_{W^i(t)},f_t \right) \right].
    \end{align*}
    The first term can be controlled directly by
    \begin{align*}
        \E \left[ \mathcal{W}_1 \left( N^{-1} \sum_{i=1}^N \delta_{Z^i(t)},N^{-1} \sum_{i=1}^N \delta_{W^i(t)} \right) \right] &\, \leq N^{-1} \sum_{i=1}^N \E |Z^i(t)-W^i(t)|\\
        &\, = N^{-1} \mathbf{W}_1 \left( F_N(t),f_t^{\otimes N} \right) \leq C_0 N^{-1/4+\varepsilon}.
    \end{align*}
    For the second term, since $W^1(t),\cdots,W^N(t)$ are i.i.d. with common law $f_t$, by the classical result of Fournier--Guillin \cite[Theorem 1]{fournier2014particle}, we have
    \begin{equation*}
        \E \left[ \mathcal{W}_1 \left( N^{-1} \sum_{i=1}^N \delta_{W^i(t)},f_t \right) \right] \leq CN^{-1/3}.
    \end{equation*}
    This finishes the proof.
\end{proof}

\subsection*{Acknowledgments}  This work was essentially completed while Chenguang Liu was visiting the Beijing International Center for Mathematical Research (BICMR), Peking University. Chenguang Liu would like to thank BICMR for its hospitality. The work of Xuanrui Feng and Zhenfu Wang was partially supported by the National Key R\&D Program of China (Project No.~2024YFA1015500) and the NSFC (Grant Nos.~12595282 and 12171009).

\bibliography{ref}
\bibliographystyle{abbrv}

\end{document}